\documentclass[12pt]{amsart}
\usepackage{amssymb}
\usepackage{mathtools}
\usepackage{txfonts}
\usepackage[thicklines]{cancel}
\usepackage{eucal}
\usepackage{extarrows}
\usepackage{capt-of}
\usepackage{shuffle}
\usepackage[all]{xy}
\usepackage{tikz-cd}
\usepackage{tikz}
\usepackage{graphicx}
\usepackage{float}
\usepackage{xspace}
\usepackage{caption}
\usepackage[a4paper,body={16.3cm,22.8cm},centering]{geometry}
\usepackage[colorlinks,final,backref=page,hyperindex,pdftex]{hyperref}
\newcommand{\nc}{\newcommand}
\newcommand{\delete}[1]{}

\nc{\mlabel}[1]{\label{#1}}  
\nc{\mcite}[1]{\cite{#1}}  
\nc{\mref}[1]{\ref{#1}}  
\nc{\mbibitem}[1]{\bibitem{#1}} 

\delete{
\nc{\mlabel}[1]{\label{#1}  
{\hfill \hspace{1cm}{\small\tt{{\ }\hfill(#1)}}}}
\nc{\mcite}[1]{\cite{#1}{\small{\tt{{\ }(#1)}}}}  
\nc{\mref}[1]{\ref{#1}{{\tt{{\ }(#1)}}}}  
\nc{\mbibitem}[1]{\bibitem[\bf #1]{#1}} 
}

\newtheorem{theorem}{Theorem}[section]
\newtheorem{prop}[theorem]{Proposition}
\newtheorem{lemma}[theorem]{Lemma}
\newtheorem{coro}[theorem]{Corollary}

\theoremstyle{definition}
\newtheorem{defn}[theorem]{Definition}
\newtheorem{remark}[theorem]{Remark}
\newtheorem{exam}[theorem]{Example}
\newtheorem{prop-def}{Proposition-Definition}[section]

\newcommand\cal[1]{\mathcal{#1}}

\newcommand\alphlist{a,b,c,d,e,f,g,h,i,j,k,l,m,n,o,p,q,r,s,t,u,v,w,x,y,z}
\newcommand\Alphlist{A,B,C,D,E,F,G,H,I,J,K,L,M,N,O,P,Q,R,S,T,U,V,W,X,Y,Z}
\newcommand\getcmds[3]{\expandafter\newcommand\csname #2#1\endcsname{#3{#1}}}
\makeatletter
\@for\x:=\alphlist\do{\expandafter\getcmds\expandafter{\x}{frak}{\mathfrak}}
\@for\x:=\Alphlist\do{\expandafter\getcmds\expandafter{\x}{frak}{\mathfrak}}
\makeatother

\nc{\bfk}{{\bf k}}
\font\cyr=wncyr10

\newfont{\scyr}{wncyr10 scaled 550}
\nc{\sha}{\mbox{\cyr X}}
\nc{\ssha}{\mbox{\bf \scyr X}}

\nc{\id}{\mathrm{id}}
\nc{\Id}{\mathrm{Id}}
\nc{\lbar}[1]{\overline{#1}}
\nc{\llbar}[1]{#1}
\nc{\ot}{\otimes}
\nc{\dep}{\mathrm{dep}}
\nc{\lt}[1]{\lbar{#1}}

\nc{\tred}[1]{\textcolor{red}{#1}} \nc{\tgreen}[1]{\textcolor{green}{#1}}
\nc{\tblue}[1]{\textcolor{blue}{#1}} \nc{\tpurple}[1]{\textcolor{purple}{#1}}

\nc{\li}[1]{\tpurple{\underline{Li:}#1 }}
\nc{\liadd}[1]{\tpurple{#1}}
\nc{\xing}[1]{\tblue{\underline{Xing:}#1 }}
\nc{\dominique}[1]{\tblue{\underline{Dominique: }#1 }}
\nc{\yuan}[1]{\tred{\underline{Yuan:}#1 }}
\nc{\markus}[1]{\tred{\underline{Markus:} #1}}
\nc{\hu}[1]{\tpurple{\underline{Huhu:}#1 }}
\nc{\Hu}[1]{\tpurple{#1 }}

\usetikzlibrary{calc,shapes.geometric}
\tikzset{
baseon/.style={baseline={($(#1)+(0,-0.58ex)$)}},
baseon/.default=current bounding box.center,
every picture/.style=baseon,
lst/.style={},
dst/.style={circle,inner sep=0.5pt,outer sep=0pt,fill,draw,dst2},
dst2/.style={fill=white},
ddst/.style={diamond,draw,inner sep=0.5pt},
eest/.style={ellipse,draw,inner sep=0.5pt,minimum size=0.5ex},
}

\newcommand\treeo[2][]{\tikz[x=0.4cm,y=0.4cm,line width=0.15ex,
every node/.style={font=\scriptsize,inner sep=0.5pt,label distance=0.5pt},#1]{%
\coordinate (o) at (0,0);#2}}%
\def\zzz#1`#2...#3`#4...#5`#6@{%
--++(#1)
node[dst,label={#5:$#6$},name=#2]{}
node[midway,auto,#3]{$#4$}
}
\def\ddd#1`#2`#3@{+(#1)node[ddst,name=#2]{$#3$}}
\def\eee#1`#2`#3@{+(#1)node[eest,name=#2]{$#3$}}
\def\xxx#1`#2@{node[midway,auto,inner sep=1pt,#1]{$#2$}}
\def\pp#1`#2`#3@{node[dst,label={#2:$#3$},pos=#1]{}}
\def\oo#1`#2`#3@{\path (o) node[dst,label={#2:$#3$},name=o,#1]{};}
\def\eoo#1`#2@{\node[eest,name=o,#1] at (o) {$#2$};}

\newif\ifshowjdq
\showjdqtrue
\newcommand\setXXclip[3]{%
\def\XXheight{#1}\def\XXdepth{#2}\def\XXwidth{#3}}
\setXXclip{1}{-0.5}{1.1}

\def\nnn#1`#2`#3@{+(#1)node[name=#2]{$#3$}}
\newcommand\oneo[3]{%
\treeo[x=0.5cm,y=0.5cm]{
\eoo`#1@
\draw (o) -- +(0,-1);
\draw (o) \nnn-1,1`l`#2@ \nnn1,1`r`#3@ (o)--(l) (o)--(r);
}}
\newcommand\twool[5]{%
\treeo[x=0.5cm,y=0.5cm]{
\eoo`#1@
\draw (o) -- +(0,-1);
\draw (o) \eee-1,1`l`#2@ \nnn1,1`r`#3@ (o)--(l) (o)--(r);
\draw (l) \nnn-1,1`ll`#4@ \nnn1,1`lr`#5@ (l)--(ll) (l)--(lr);
}}
\newcommand\twoor[5]{%
\treeo[x=0.5cm,y=0.5cm]{
\eoo`#1@
\draw (o) -- +(0,-1);
\draw (o) \nnn-1,1`l`#2@ \eee1,1`r`#3@ (o)--(l) (o)--(r);
\draw (r) \nnn-1,1`rl`#4@ \nnn1,1`rr`#5@ (r)--(rl) (r)--(rr);
}}
\newcommand\threeoll[7]{%
\treeo[x=0.4cm,y=0.4cm]{
\eoo`#1@
\draw (o) -- +(0,-1);
\draw (o) \eee-1,1`l`#2@ \nnn1,1`r`#3@ (o)--(l) (o)--(r);
\draw (l) \eee-1,1`ll`#4@ \nnn1,1`lr`#5@ (l)--(ll) (l)--(lr);
\draw (ll) \nnn-1,1`lll`#6@ \nnn1,1`llr`#7@ (ll)--(lll) (ll)--(llr);
}}
\newcommand\threeolr[7]{%
\treeo[x=0.4cm,y=0.4cm]{
\eoo`#1@
\draw (o) -- +(0,-1);
\draw (o) \eee-1,1`l`#2@ \nnn1,1`r`#3@ (o)--(l) (o)--(r);
\draw (l) \nnn-1,1`ll`#4@ \eee1,1`lr`#5@ (l)--(ll) (l)--(lr);
\draw (lr) \nnn-1,1`lrl`#6@ \nnn1,1`lrr`#7@ (lr)--(lrl) (lr)--(lrr);
}}
\newcommand\threeorr[7]{%
\treeo[x=0.4cm,y=0.4cm]{
\eoo`#1@
\draw (o) -- +(0,-1);
\draw (o) \nnn-1,1`l`#2@ \eee1,1`r`#3@ (o)--(l) (o)--(r);
\draw (r) \nnn-1,1`rl`#4@ \eee1,1`rr`#5@ (r)--(rl) (r)--(rr);
\draw (rr) \nnn-1,1`rrl`#6@ \nnn1,1`rrr`#7@ (rr)--(rrl) (rr)--(rrr);
}}
\newcommand\threeorl[7]{%
\treeo[x=0.4cm,y=0.4cm]{
\eoo`#1@
\draw (o) -- +(0,-1);
\draw (o) \nnn-1,1`l`#2@ \eee1,1`r`#3@ (o)--(l) (o)--(r);
\draw (r) \eee-1,1`rl`#4@ \nnn1,1`rr`#5@ (r)--(rl) (r)--(rr);
\draw (rl) \nnn-1,1`rll`#6@ \nnn1,1`rlr`#7@ (rl)--(rll) (rl)--(rlr);
}}
\newcommand\threeoY[7]{%
\treeo[x=0.4cm,y=0.4cm]{
\eoo`#1@
\draw (o) -- +(0,-1);
\draw (o) \eee-1,1`l`#2@ \eee1,1`r`#3@ (o)--(l) (o)--(r);
\draw (l) \nnn-1,1`ll`#4@ \nnn0.5,1`lr`#5@ (l)--(ll) (l)--(lr);
\draw (r) \nnn-0.5,1`rl`#6@ \nnn1,1`rr`#7@ (r)--(rl) (r)--(rr);
}}

\makeatletter
\newcommand\simra{\mathrel{\mathpalette\@verra\sim}}
\def\@verra#1#2{\lower.5\p@\vbox{\lineskiplimit\maxdimen \lineskip-.5\p@
\ialign{$\m@th#1\hfil##\hfil$\crcr#2\crcr\rightarrow\crcr}}}
\makeatother

\nc{\dnx}{\Delta_n A} \nc{\dx}{\Delta A} \nc{\dgp}{{\rm deg_{P}}}
\nc{\dgt}{{\rm deg_{T}}} \nc{\dg}{{\rm deg}} \nc{\ida}{ID($A$)} \nc{\tu}{\tilde{u}} \nc{\tv}{\tilde{v}}
\nc{\nr}{\calr_n} \nc{\nz}{\calz_n} \nc{\fun}{\cala_{n,d}}
 \nc{\fbase}{\calb} \nc{\LF}{\mathrm{RF}} \nc{\FFA}{\mathrm{LF}} \nc{\irr}{\mathrm{Irr}}
 \nc{\result}{\bfk\mathrm{Irr}(S_n)}  \nc{\I}{I_{\mathrm{ID},n}^0}
 \nc{\nrs}{\calr_n^\star} \nc{\ii}{\mathrm{I}} \nc{\iii}{\mathrm{II}}
\nc{\intl}{{\rm int}}\nc{\ws}[1]{{#1}}\nc{\deleted}[1]{\delete{#1}}\nc{\plas}{placements\xspace}

\nc{\bim}[1]{#1}  \nc{\shaop}{\sha_{\Omega}^{+}}  \nc{\shao}{\sha_{\Omega}}
\nc{\bbim}[2]{#1 #2} \nc{\bbbim}[2]{#1,\, #2} \nc{\RBF}{{\rm RBF}}
\nc{\frb}{F_{\RB}} \nc{\shaf}{\ssha_{\tiny{\Omega}}} \nc{\sham}{\diamond_{\tiny{\Omega}}}
\nc{\lf}{\lfloor} \nc{\rf}{\rfloor} \nc{\shan}{\ssha_{\lambda}}
\nc{\rlex}{{\rm {lex}}} \nc{\bb}{\Box} \nc{\ra}{\rightarrow}
\nc{\e}{{\rm {e}}}
\nc{\DDF}{\mathrm{DD}(X,\,\Omega)}\nc{\DTF}{\mathrm{DT}(X,\,\Omega)} \nc{\DT}{\mathrm{DT}'(\Omega,\,V)}
\nc{\bra}{\mathrm{bra}} \nc{\bre}{\mathrm{bre}}
\nc{\dec}{\mathrm{dec}} \nc{\diamondw}{\diamond_{w}}
\nc{\type}{\mathrm{type}}

\nc\caF[1]{\cal{F}_{#1}(X,\,\Omega)}
\nc\calt{\cal{T}} \nc\caltn{\cal{T}_n(X,\,\Omega)}
\nc\caltbin{\cal{T}_b(X,\,\Omega)}
\nc\calta{\cal{T}_0(X,\,\Omega)}
\nc\caltb{\cal{T}_1(X,\,\Omega)}
\nc\caltc{\cal{T}_2(X,\,\Omega)}
\nc\caltd{\cal{T}_3(X,\,\Omega)}
\nc\caltm{\cal{T}_m(X,\,\Omega)}
\nc\calf{\cal{F}(X,\,\Omega)}
\nc\fram{\frak{M}(\Omega,\, X)}
\nc\shaw{\sha^{NC}_w(\Omega,\, X)}
\nc\dw{\diamond_w} \nc\dl{\diamond_\ell}
\nc\shal{\sha^{NC}_\ell(X,\, \Omega)} \nc\shav{\sha^{NC}_w(\Omega,\, V)} \nc\shat{\sha^{NC,1}_w(\Omega,\, T^{+}(V))}
\nc{\cfo}{\cal{F}(X,\,\Omega)}

\nc{\lar}{\varinjlim}
\nc\XO{(X,\,\Omega)}
\def\cxo#1#2;{\cal{#1}#2\XO}
\def\cxob#1#2;{\cal{#1}#2_b\XO}
\nc\lrf[2]{B_{#2}^+(#1)}
\nc{\fd}{\mathrm{\text{typed angularly decorated planar rooted trees}}}
\nc{\rb}{\mathrm{RBFWs}} \nc{\dfw}{\mathrm{DFW{(X)}}} \nc{\tfw}{\mathrm{TFW{(X)}}}
\nc{\tfv}{\mathrm{TFW{(V)}}} \nc{\rbf}{\mathrm{RBF}}

\def\Ve#1,#2,#3;{\vee_{#1,\,(#2,\,#3)}}
\def\bigv#1;#2;#3;{\bigvee\nolimits_{#1}^{#2;\,#3}}

\nc{\Irr}{\mathrm{Irr}}
\nc{\Tsucs}{{trisuccessors}\xspace}

\nc{\gensp}{\mathcal{V}} 
\nc{\relsp}{R} 
\nc{\leafsp}{\mathcal{X}}    
\nc{\treesp}{\mathbb{T}} 
\nc{\genbas}{X} 
\nc{\opd}{\mathcal{P}} 

\nc{\vin}{\mathrm{V}}    
\nc{\lin}{\mathrm{L}}    
\nc{\inv}{\mathrm{In}}

\nc{\bvp}{V_P}     

\nc{\gop}{{\,\omega\,}}     
\nc{\gopb}{{\,\nu\,}}
\nc{\svec}[2]{{\textrm{\tiny{$\left(\begin{matrix}#1\\
#2\end{matrix}\right)$}}}}  
\nc{\ssvec}[2]{{\textrm{\tiny{$\left(\begin{matrix}#1\\
#2\end{matrix}\right)$}}}} 
\nc{\treeg}[5]{\vcenter{\xymatrix@M=1.5pt@R=1.5pt@C=0pt{#1 & & #2 & & & #3 \\ & #5 \ar@{-}[lu] \ar@{-}[ru] & & & & \\ & & #4 \ar@{-}[lu] \ar@{-}[rrruu] & & \\ & & & & & \\ & & \ar@{-}[uu] & & & }}}
\nc{\treed}[5]{\vcenter{\xymatrix@M=2pt@R=4pt@C=2pt{#1 & & & #2  & & #3 \\ & & & & #5 \ar@{-}[lu] \ar@{-}[ru] & \\ & & & #4 \ar@{-}[ru] \ar@{-}[llluu] & \\ & & & & & \\ & & & \ar@{-}[uu] & & }}}

\nc{\tsvec}[3]{{\textrm{\tiny{$\left(\begin{matrix}#1\\
#2\\#3\end{matrix}\right)$}}}}  

\nc{\stsvec}[3]{{\textrm{\tiny{$\left(\begin{matrix}#1\\
#2\\#3\end{matrix}\right)$}}}} 

\nc{\su}{\mathrm{DSu}}
\nc{\tsu}{\mathrm{TSu}}
\nc{\TSu}{\mathrm{TSu}}
\nc{\eval}[1]{{#1}_{\big|D}}
\nc{\oto}{\leftrightarrow}

\nc{\oaset}{\mathbf{O}^{\rm alg}}
\nc{\omset}{\mathbf{O}^{\rm mod}}
\nc{\oamap}{\Phi^{\rm alg}}
\nc{\ommap}{\Phi^{\rm mod}}
\nc{\ioaset}{\mathbf{IO}^{\rm alg}}
\nc{\iomset}{\mathbf{IO}^{\rm mod}}
\nc{\ioamap}{\Psi^{\rm alg}}
\nc{\iommap}{\Psi^{\rm mod}}

\nc{\suc}{{disuccessor}\xspace} \nc{\Suc}{{Disuccessor}\xspace}
\nc{\sucs}{{disuccessors}\xspace} \nc{\Sucs}{{Disuccessors}\xspace}
\nc{\Tsuc}{{trisuccessor}\xspace}
\nc{\TSuc}{{Trisuccessor}\xspace}
\nc{\TSucs}{{Trisuccessors}\xspace}
\nc{\Lsuc}{{L-successor}\xspace}
\nc{\Lsucs}{{L-successors}\xspace} \nc{\Rsuc}{{R-successor}\xspace}
\nc{\Rsucs}{{R-successors}\xspace}

\nc{\bia}{{$\mathcal{P}$-bimodule ${\bf k}$-algebra}\xspace}
\nc{\bias}{{$\mathcal{P}$-bimodule ${\bf k}$-algebras}\xspace}

\nc{\rmi}{{\mathrm{I}}}
\nc{\rmii}{{\mathrm{II}}}
\nc{\rmiii}{{\mathrm{III}}}

\nc{\pll}{\beta}
\nc{\plc}{\epsilon}
\nc\np{\mathcal{P}}
\nc{\BS}{{\mathbb{S}}}
\nc{\prelie}{{\mathit{Pre\text{-}Lie}}}
\nc{\LDend}{{\mathit{LDend}}}
\nc{\perm}{{\mathit{Perm}}}
\nc{\as}{{\mathit{As}}}
\nc{\lie}{{\mathit{Lie}}}
\nc{\dend}{{\mathit{Dend}}}
\nc{\qua}{{\mathit{Quad}}}
\nc{\lquad}{{\mathit{LQuad}}}
\nc{\supp}{\rm Supp}
\nc{\dps}{\dotplus}
\nc{\tvarrow}[3]{#1 \overset{(t,v)}{\longrightarrow}_{#3} #2}\nc{\gs}{Gr\"{o}bner-Shirshov\xspace}
\nc{\astarrow}{\overset{\raisebox{-2pt}{{\scriptsize $\ast$}}}{\rightarrow}}
\nc\oplie{{\rm OLie}(X)}\nc\sopma[1]{\mathfrak{M}^\star(#1)}\nc\Pia[1]{\Pi_{#1}^{\rm ass}}\nc\Pil[1]{\Pi_{#1}^{\rm Lie}}
\nc\plie[1]{\mathfrak{S}(#1)}\nc\plien[1]{\mathcal{N}(#1)}\nc\nas[1]{{#1}^\ast}
\nc\sopm[1]{\mathfrak{S}^\star(#1)}
\nc{\suba}[1]{|_{#1}}\nc\ip{\operatorname{In}}\nc\PBTN{{\rm PBT}}\nc\PMTN{{\rm TD}}
\nc\Omegax{X}\nc\fna{L(X)}\nc\fma{{{\rm Mag}_2(X)}}\nc\lp{\succ}\nc\rp{\prec}\nc\lde{{\rm LW}(X)}
\nc\lw{{\rm LT}(X)}\nc\lrea[3]{#1 \prec\left(#2 \ast #3\right)+\left( #2 \succ  #1\right) \prec  #3- #2 \succ\left( #1 \prec  #3\right)}
\nc\R{K} \nc{\ildend}{I_{{\rm LDend}}}
\nc\circa{\vee}
\nc{\bcdot}{m}

\nc\gleft{\swarrow_\lp}
\nc\gright{\swarrow_\rp}
\nc\gd{\swarrow}
\begin{document}

\title[Free L-dendriform algebras]{Disuccessors, Gr\"obner-Shirshov bases and free L-dendriform algebras}

\author{Huhu Zhang}
\address{School of Mathematics and Statistics
	Yulin University, Yulin, 719000, China}
\email{huhuzhang@yulinu.edu.cn}

\author{Xing Gao$^*$}\thanks{*Corresponding author}
\address{School of Mathematics and Statistics, Lanzhou University,
Lanzhou, 730000, China; Gansu Provincial Research Center for Basic Disciplines of Mathematics and Statistics, Lanzhou, 730070, China}
\email{gaoxing@lzu.edu.cn}

\author{Yuanyuan Zhang}
\address{School of Mathematics and Statistics, Henan University, Kaifeng, 475004, China}
\email{zhangyy17@henu.edu.cn}

\date{\today}

\begin{abstract}
In this paper, we prove respectively that the disuccessor operations on the associative operad $\as$, the Lie operad $\lie$, and the pre-Lie operad $\prelie$ preserve Gr\"obner-Shirshov bases. This structural preservation enables the transfer of known bases to more complex operads.
As a consequence, we introduce new methods for constructing Gr\"obner-Shirshov bases for the $\dend$ and $\prelie$ operads, and explicitly construct a Gr\"obner-Shirshov basis for the free L-dendriform algebra. This provides a conceptual and computationally efficient resolution of Madariaga's problem and offers new insights into the combinatorial structure of L-dendriform algebras.
\end{abstract}

\makeatletter
\@namedef{subjclassname@2020}{\textup{2020} Mathematics Subject Classification}
\makeatother
\subjclass[2020]{
16W99, 
16S10, 
13P10, 
08B20, 
17D25      
}

\keywords{L-dendriform algebras, pre-Lie  algebras, Lie algebras, free objects, Gr\"obner-Shirshov bases}

\maketitle

\tableofcontents

\setcounter{section}{0}

\allowdisplaybreaks
\section{Introduction}
In this paper, we address a problem posed by Madariaga concerning the construction of Gr\"obner-Shirshov bases for $\su^n(\lie)$-algebras, where $\su$ denotes the disuccessor operation.
We demonstrate that the disuccessor operation $\su$ preserves Gr\"obner-Shirshov bases when applied to the associative operad $\as$, the Lie operad $\lie$, and the pre-Lie operad $\prelie$. This structural property yields new methods for constructing Gr\"obner-Shirshov bases for the $\dend$ and $\prelie$ operads, and leads to an explicit Gr\"obner-Shirshov basis for the $\LDend$ operad. Our results provide a conceptual and computationally efficient solution to Madariaga's problem in the cases $n = 1$ and $n = 2$.

\subsection{From dendriform algebras to L-dendriform algebras}
The concept of a dendriform algebra was introduced by Jean-Louis Loday in 1995~\cite{Lod01}, originally motivated by questions arising in algebraic $K$-theory. Since its inception, the theory of dendriform algebras has developed significantly and found deep connections across a broad spectrum of mathematics and physics. These include, but are not limited to, the theory of operads, homology and homotopy theory, Hopf algebras, Lie and Leibniz algebras, combinatorics, arithmetic geometry, and even aspects of quantum field theory (see~\cite{EMP} and references therein).

One of the striking features of dendriform algebras is how they refine associative structures by decomposing the associative product into two binary operations satisfying certain axioms. This decomposition allows for a richer algebraic structure that interfaces naturally with other algebraic frameworks. In particular, there exists an operadic relationship between the operads of Lie algebras, associative algebras, pre-Lie algebras, and dendriform algebras~\cite{Aguiar, Chap, Ronco}:
$$
\xymatrix{
\rm{dendriform ~~operad} & \rm{associative ~~operad}  \ar[l]_{\text{splitting}}\\
\rm{pre\text{-}Lie~ ~operad}\ar[u]^{\text{commutator}}& \rm{Lie ~~operad} \ar[l]_{~~\text{splitting}}^{\text{ or commutator}}\ar[u]^{\text{commutator}}.
}
$$

L-dendriform algebras~\cite{BLN10} act as the Lie-algebraic analogues of dendriform algebras, providing the necessary structure to ensure the commutativity of the extended diagram below:
$$
\xymatrix@C=0.4cm@R=0.4cm{
\cdots\quad\rm{Quadri\text{-}operad}&\rm{dendriform ~~operad} \ar[l]_{\text{splitting}} & \rm{associative ~~operad} \ar[l]_{\text{splitting}}\\
\cdots&&& \rm{Lie ~~operad}\ar[ld]_{~~\text{splitting}}^{\text{ ~~~~~~~or commutator}} \ar[lu]_{~~\text{commutator}}\\
&\cdots\quad\rm{\text{L-}dendriform~~operad}\ar[luu]_{~~\text{commutator}}&\rm{pre\text{-}Lie~ ~operad}\ar[l]_{\text{\quad\quad\quad\quad splitting}}  \ar[luu]_{~~\text{commutator}}&.
}
$$
L-dendriform algebras are motivated from the underlying algebraic structure associated with a pseudo-Hessian pre-Lie algebra or a pseudo-Hessian Lie algebra~\cite{BLN10}. This construction is rooted in differential geometry, particularly in the study of Hessian manifolds.
A Hessian manifold $M$ is a smooth manifold equipped with both a flat affine connection and a Riemannian metric $g$ satisfying a local potential condition~\cite{Shima}: for every point $p \in M$, there exists a smooth function $\phi \in C^\infty$ defined on a neighborhood of $p$ such that
\[
g_{ij} = \frac{\partial^2 \phi}{\partial x^i \partial x^j}.
\]
In other words, $g$ is locally given by the Hessian of a potential function. Such manifolds naturally arise in information geometry and affine differential geometry.
On the algebraic side, the structure corresponding to an affine Lie group $G$ equipped with a $G$-invariant Hessian metric is a real pre-Lie algebra endowed with a symmetric, positive definite $2$-cocycle. For an overview of pre-Lie algebras and their applications, see~\cite{Burde}.
The classical notion of Hessian geometry can be generalized to the pseudo-Hessian setting by replacing the requirement of positive definiteness with nondegeneracy of the bilinear form. This relaxation allows one to consider indefinite metrics and opens the door to extending these geometric ideas to other fields (beyond $\mathbb{R}$) at the purely algebraic level. In this broader framework, L-dendriform algebras serve as a unifying structure capturing the interaction between pre-Lie algebras and pseudo-metric data~\cite{BLN10,NB14}.

\subsection{Splitting of operads}
Since the late 1990s, a variety of algebraic structures defined by multiple binary operations have emerged, beginning with the introduction of dendriform algebras and tridendriform algebras by Loday~\cite{LoRo04,Lod01}. These foundational structures inspired the development of a broad family of related algebras, including L-dendriform algebras, ennea-algebras, NS-algebras, dendriform-Nijenhuis algebras, and octo-algebras, among others~\cite{Ler04, Leroux, Ler03}. These algebras can often be interpreted as splitting of binary operations inherent in more classical or unified algebraic structures.

At the level of operads, Vallette~\cite{Va} introduced two new binary operations---Manin black and white products, analogous to the classical Manin black and white products for algebras, and provided explicit computational techniques for evaluating them. These constructions revealed a deep relationship between operadic splitting and the Manin black product with the pre-Lie operad $\prelie$. In particular, the Manin black product $\prelie \bullet \mathcal{P}$, where $\mathcal{P}$ is another binary operad, frequently results in a refined operad whose algebras exhibit split versions of the operations in $\mathcal{P}$. Some notable examples of this phenomenon include:
\begin{align*}
\prelie \bullet \as = \dend, \quad \prelie \bullet \lie = \prelie, \quad \prelie \bullet \prelie = \LDend,
\end{align*}
where $\LDend$ is the operad governing L-dendriform algebras.

In a significant advancement, Guo et al.~\cite{BBG13} established a general operadic framework that formalizes the notion of splitting of binary operads via the concept of successors. Specifically, they introduced the disuccessor and trisuccessor of a binary quadratic operad, defined in terms of generators and relations, and showed how this construction generalizes both the operadic Manin black product and Rota-Baxter type splitting mechanisms.
These developments reveal a rich web of interconnections among classical and generalized algebraic operads. In particular, there exists a commutative diagram of operads that encodes the relationships among
\[
\as, \quad \lie, \quad \prelie, \quad \LDend, \quad \dend, \quad \qua \text{ (the operad of quadri-algebras~\cite{AL})},
\]
illustrating how successive operadic constructions give rise to increasingly structured algebraic systems.
$$
\xymatrix@C=0.5cm@R=0.4cm{
\lie \ar[d]_{\rm splitting} \ar[rr]^{[~]\mapsto \cdot-\cdot(12)}&& \as \ar[d]^{\rm splitting}\\
\su(\lie)=\prelie\ar[d]_{\rm splitting}\ar[rr]^{\bullet \mapsto \succ-\prec}&& \dend=\su(\as)\ar[d]^{\rm splitting}\\
\su^2(\lie)=\su(\prelie)=\LDend \ar[d]_{\rm splitting} \ar[rr]^{\quad \succ \mapsto \searrow-\nwarrow}_{\quad \prec\mapsto \nearrow-\swarrow }&& \qua=\su(\dend)=\su^2(\as) \ar[d]^{\rm splitting}\\
\su^n(\lie), n\geq3\ar[rr]&&\su^n(\as), n\geq3
}
$$

\subsection{Gr\"obner-Shirshov bases}
The theory of Gr\"obner-Shirshov bases for Lie algebras was pioneered by Shirshov~\cite{Shirshov}, who established the foundational Composition-Diamond lemma in this context.
Later, Bokut~\cite{B76} extended Shirshov's approach to associative algebras, marking a substantial generalization of the method.
This broader framework laid the groundwork for the subsequent development of the theory of Gr\"obner bases in diverse algebraic settings.
Bergman~\cite{Bergman} further advanced this direction, refining the methods and contributing key insights into Gr\"obner basis theory for associative algebras.
In the realm of commutative polynomials, the Composition-Diamond lemma manifests as Buchberger's Theorem~\cite{Buch65, Buch70}, a cornerstone of Gr\"obner basis theory, computational algebraic geometry, and polynomial ideal theory.
Building on these foundational works, Drensky and Holtkamp~\cite{DH} extended the Shirshov's Composition-Diamond lemma to the setting of free algebras involving operations of arity two. This advancement bridged the classical results with more generalized algebraic structures, paving the way for further exploration in operadic and non-associative algebraic frameworks using Gr\"obner basis techniques.
It is worth mentioning that Gr\"obner basis theory and Gr\"obner-Shirshov basis theory are equivalent in the context of associative algebras. This equivalence not only highlights the versatility of the theory but also demonstrates its foundational importance in both computational and theoretical algebra.

An analogous extension to the realm of operads has been developed in~\cite{DK09, LV12}, where a systematic and comprehensive framework for Gr\"obner bases in the operadic context was established. These works provided tools for studying the combinatorial and algebraic properties of operads, enabling researchers to address complex questions about ideal theory, syzygies, and resolutions in this broader algebraic setting. This operadic generalization not only connects classical Gr\"obner basis theory to modern contexts but also opens avenues for exploring new applications in homotopical and categorical algebra.

\subsection{Motivation and outline of the paper}
Madariaga posed the following open problems~\cite{Mada14}:
\begin{quote}
If we iterate this splitting procedure, we obtain a series of structures with $2^n$ binary nonassociative
 operations, namely $\su^n(\lie)$-algebras. It would be interesting to have Gr\"obner-Shirshov bases for these structures,
 as well as for the algebras defined by the dual operads.
\end{quote}
This motivates our investigation into the theory of Gr\"obner-Shirshov bases for the operads $\su^n(\lie)$ with $n \geq 1$. However, this study faces significant challenges. On the one hand, the structures involved contain multiple operations and relations, making the computation of Gr\"obner-Shirshov bases highly intricate. On the other hand, the disuccessor operation $\su$ does not always preserve the Gr\"obner-Shirshov property. For example, the operad $\su({\mathcal{N}il})$ is known not to be Koszul~\cite{Va}, and by~\cite[Theorem 8.3.1]{LV12}, its defining relations do not form a Gr\"obner-Shirshov basis.
In other words, there exist operads $\mathcal{P}$ for which the disuccessor operation $\su$ fails to preserve Gr\"obner-Shirshov bases. For this reason, the present work focuses only on those operads where the disuccessor operations do preserve such bases. In particular, we study the cases
\[
\su(\lie) = \prelie, \quad  \su^2(\lie) = \su(\prelie) = \LDend.
\]
As far as we know, this is the first time that the relationship between the successor operation and Gr\"obner-Shirshov bases has been systematically studied.

\noindent {\bf Outline of the paper.}
In Section~\ref{sec:FRBF}, we recall several foundational concepts and preliminary results that are essential for our study of operadic Gr\"obner-Shirshov bases.

Section~\ref{sec:bsoperad} is devoted to demonstrate that the disuccessors of the associative operad, the Lie operad, and the pre-Lie operad each possess a Gr\"obner-Shirshov basis.
In more details, in Subsection~\ref{sec:As operad}, we use the associative operad $\as$ as a case study to illustrate the concept of disuccessors of binary operads.
We show that the disuccessor of $\as$---which yields the dendriform operad $\dend$---preserves Gr\"obner-Shirshov bases (Theorem~\ref{thm:gsbdend}).
This provides a new approach to constructing a Gr\"obner-Shirshov basis for the operad $\dend$.
Subsection~\ref{sec:lie operad} focuses on the Lie operad $\lie$ and its disuccessor, the pre-Lie operad $\prelie$. We give a new method to present a Gr\"obner-Shirshov basis for the $\prelie$ operad and discuss how successor constructions apply in this context (Theorem~\ref{thm:lipre}).
In Subsection~\ref{sec:prel operad}, we study the disuccessor of the $\prelie$ operad, namely the $\LDend$ operad corresponding to L-dendriform algebras. We construct a Gr\"obner-Shirshov basis for the $\LDend$ operad (Theorem~\ref{thm:gsbases} ) and use it to derive the free L-dendriform algebra on a given set.

In Section~\ref{sect:free}, we construct three linear bases for the free L-dendriform algebra (Theorem~\mref{thm:gsbld}): one obtained via the Gr\"obner-Shirshov basis method, and two others, $\lde$ and $\lw$, derived by taking the free two-magma algebra as an intermediate step and employing special classes of typed decorated planar rooted trees.

\smallskip
\noindent
\textbf{Notation.} Throughout this paper, we fix a base field $\bfk$, over which all vector spaces, algebras, tensor products, and linear maps are defined. Unless otherwise specified, by an algebra we mean a unital, associative, and possibly noncommutative algebra over $\bfk$. We denote by $S_n$ the symmetric group on $n$ elements, for $n \geq 1$, and by $\mathbb{N}$ the set of nonnegative integers.

\section{Operads, Gr\"obner-Shirshov bases and disuccessors}
\mlabel{sec:FRBF}
In this section, we introduce the preliminary notions and key algebraic tools that will serve as the foundation for the results developed in the subsequent sections. Our focus is primarily on operads, Gr\"obner-Shirshov bases and disuccessors of binary operads.

\subsection{Nonsymmetric operads and symmetric operads}
\mlabel{sub:basic definitions}
This subsection is dedicated to recalling the fundamental definitions and related concepts of operads.

\begin{defn}\cite[Section 5.1]{LV12}
  An {\bf $ \mathbb{S}$-module (resp.  arity-graded module) } over $\bfk$ is a family
\[M=\{M(0), M(1), \ldots, M(n),\ldots\}, \quad (\text{ resp. } M=\{M_0, M_1, \ldots, M_n,\ldots\} )\]
of right $\bfk[\mathbb{S}_n]$-modules $M(n)$ (resp. $\mathbb{N}$-modules $M_n$). When $M(0) = 0$, the $\mathbb{S}$-module (resp. $\mathbb{N}$-module) is called {\bf reduced}.
\mlabel{defn:smodule}
\end{defn}

\begin{defn}~\cite[Section 5.8.4]{LV12}
{\bf A nonsymmetric operad} is an arity-graded module $\np = \{\np_n\}_{n\geq 0}$ equipped with partial compositions:
\[\circ_i:\np_m\ot\np_n\ra \np_{m-1+n}, \quad\text{ for }\,1\leq i\leq m,\]
satisfying the relations
$$ \left\{
\begin{aligned}
(\lambda\circ_i \mu)\circ_{i-1+j} v = & \  \lambda\circ_i(\mu\circ_j v),\quad \text{ for } 1\leq i\leq l, 1\leq j\leq m, \\
(\lambda\circ_i \mu)\circ_{k-1+m} v = & \ (\lambda\circ_k v)\circ_i\mu,\quad\text{ for }\,1\leq i\leq k\leq l,
\end{aligned}
\right.
$$
for any $\lambda\in\np_l, \mu\in\np_m, v\in\np_n.$
\end{defn}

\begin{defn}~\cite[Definition 3, Section 5.1.0]{LV12}
{\bf A symmetric operad} is an $\mathbb{S}$-modules $\np = \{\np(n)\}_{n\geq 0}$ equipped with partial compositions
\[\circ_i:\np(m)\ot\np(n)\ra \np(m-1+n), \quad \text{ for }\,1\leq i\leq m,\]
satisfying equivariance with respect to the symmetric groups and the axioms:
$$ \left\{
\begin{aligned}
(\lambda\circ_i \mu)\circ_{i-1+j} v = & \  \lambda\circ_i(\mu\circ_j v),\quad \text{ for } 1\leq i\leq l, 1\leq j\leq m, \\
(\lambda\circ_i \mu)\circ_{k-1+m} v = & \ (\lambda\circ_k v)\circ_i\mu,\quad\text{ for }\,1\leq i\leq k\leq l,
\end{aligned}
\right.
$$
for any $\lambda\in\np(l), \mu\in\np(m), v\in\np(n).$
\end{defn}

We now recall several foundational concepts related to planar rooted trees, which will play a central role in our combinatorial constructions and operadic interpretations throughout the paper.

\begin{defn}({\bf Decorated trees})~\cite[Definition 2.1]{BBG13}
\begin{enumerate}
\item Let $\mathcal{T}$ denote the set of planar rooted trees together with the trivial tree $|$. If $t \in \mathcal{T}$ has $n$ leaves, we call $t$ an {\bf $n$-tree.} For each vertex $v$ of $t$, we let $\ip(v)$ denote the set of inputs of $v$.

\item  Let $\mathcal{V}=\{\mathcal{V}(n)\}_{n\geq1}$ be a family of sets and let $t$ be an $n$-tree. By a {\bf decorated tree (or tree monomial)} we mean a planar rooted tree $t$ together with a label on the leaves of $t$ by distinct positive integers and a decoration on each vertex $v$ of $t$ by the element of $\mathcal{V}(|\operatorname{In}(v)|)$. Let $t(\mathcal{V})$ denote the set of decorated trees for $t$ and denote by
$$
\mathcal{T}(\mathcal{V})=\coprod_{t \in \mathcal{T}} t(\mathcal{V}) .
$$
If $ t \in t(\mathcal{V})$ is an $n$-tree $t$, we call $ t$ a {\bf decorated $n$-tree}. Denote by $\mathcal{T}(\mathcal{V})(n)$ the set of all decorated $n$-trees.

\item  For $ t \in \mathcal{T}(\mathcal{V})$, we let $\operatorname{V}( t)$ (resp. $\operatorname{L}( t)$ ) denote the set of decorations (resp. label) of the vertices (resp. leaves) of $ t$.
\end{enumerate}
\end{defn}

Given a decorated tree, the label of the leaves extends naturally to a label of all edges as follows:
each edge, viewed as the output of a vertex or a leaf $v$, is assigned the label ${\rm min}(v)$.
Here, ${\rm min}(v)$ denotes the minimum among the labels of the inputs to $v$; for a leaf
$v$, this is simply the label originally assigned to $v$.

\begin{defn}({\bf Shuffle tree monomial})~\cite{BD}.
A tree monomial is called a {\bf shuffle tree monomial} if its leaf label satisfies the local increasing condition: at each vertex, the input labels, read from left to right, form an increasing sequence.
\end{defn}

See Example~\mref{ex:st} below for examples of shuffle tree monomials.
Denote by $\mathcal{T}_{\shuffle}(\mathcal{V})(n)$ the set of all shuffle $n$-tree monomials, and denote by
$$
\mathcal{T}_{\shuffle}(\mathcal{V}) := \bigsqcup_{n\geq1} \mathcal{T}_{\shuffle}(\mathcal{V})(n).
$$

\begin{exam}\mlabel{ex:st}
Let $\mathcal{V}=\{\mathcal{V}(n)\}_{n \geq 1}=\{\mathcal{V}(2)\}$ be a family of sets. Then all shuffle 3-tree monomials are
\[
\treeo[x=0.5cm,y=0.5cm]{
\eoo`\nu@
\draw (o) -- +(0,-1);
\draw (o) \eee-1,1`l`\mu@ \nnn1,1`r`3@ (o)--(l) (o)--(r);
\draw (l) \nnn-1,1`ll`1@ \nnn1,1`lr`2@ (l)--(ll) (l)--(lr);
\node[left]at(-0.1,-0.5){$1$};
\node[left]at(-0.4,0.25){$1$};
\node[right]at(0.6,0.3){$3$};
\node[left]at(-1.6,1.4){$1$};
\node[right]at(-0.4,1.4){$2$};
}\quad \text { and }\quad
\treeo[x=0.5cm,y=0.5cm]{
\eoo`\nu@
\draw (o) -- +(0,-1);
\draw (o) \nnn-1,1`l`1@ \eee1,1`r`\mu@ (o)--(l) (o)--(r);
\draw (r) \nnn-1,1`rl`2@ \nnn1,1`rr`3@ (r)--(rl) (r)--(rr);
\node[left]at(-0.1,-0.5){$1$};
\node[left]at(-0.4,0.25){$1$};
\node[right]at(0.6,0.3){$2$};
\node[right]at(1.6,1.4){$3$};
\node[right]at(-0.2,1.4){$2$};
}
\text{,\quad where }\, \mu,\nu\in \mathcal{V}(2).
\]
\end{exam}

\begin{lemma}\cite{BD}
The  $\mathbb{S}$-module $\bfk\mathcal{T}(\mathcal{V}) := \{\bfk\mathcal{T}(\mathcal{V})(n)\}_{n\geq1}$, together with the partial composition product  given by the graft
ing of rooted trees and the natural embedding $i : \mathcal{V} \to\bfk\mathcal{T}(\mathcal{V})$, is the free symmetric operad on  a family  $\mathcal{V}=\{\mathcal{V}(n)\}_{n\geq1}$ of sets.
\end{lemma}

For any $\mathbb{S}$-module  $\np = \{\np(n)\}_{n\geq 1}$, we denote by $\np^f$ the underlying arity-graded module:
$$\Big(\np^f\Big)_n:=\np(n),\quad n\geq 1,$$
which gives the forgetful functor from the category of $\mathbb{S}$-modules(or  $\mathbb{S}$-set) to the category of $\mathbb{N}$-modules(or  $\mathbb{N}$-set).

\begin{lemma}\cite{BD}
The  $\mathbb{N}$-module $\bfk\mathcal{T}_{\shuffle}(\mathcal{V}):= \{\bfk\mathcal{T}_{\shuffle}(\mathcal{V})(n)\}_{n\geq1}$ is isomorphism to  $\bfk\mathcal{T}(\mathcal{V})^f$,  and the $\bfk\mathcal{T}_{\shuffle}(\mathcal{V})$, together with the partial composition product given by the grafting of rooted trees and the natural embedding $i : \mathcal{V}^f \to\bfk\mathcal{T}_{\shuffle}(\mathcal{V})$, is the free shuffle operad on an $\mathbb{S}$-sets $\mathcal{V}=\{\mathcal{V}(n)\}_{n\geq1}$.
\end{lemma}
\nc\fnsop{\mathcal{T}^{\rm ns}}
Let $\mathcal{V}=\{\mathcal{V}_n\}_{n\geq1}$ be a family  of sets.  Denote by $\fnsop(\mathcal{V})_n$ the set of all decorated $n$-trees whose leaves are labeled from 1 to $n$ from left to right. Then  $\fnsop(\mathcal{V})_n$ is a subset of  $\mathcal{T}(\mathcal{V})(n)$ and $\mathcal{T}_\shuffle(\mathcal{V})(n)$.
Then
$$
\bfk\fnsop(\mathcal{V}) := \{\bfk\fnsop(\mathcal{V})_n\}_{n\geq 1},
$$
together with the natural embedding, is the free nonsymmetric operad~\cite{LV12} on a family  $\mathcal{V}=\{\mathcal{V}_n\}_{n\geq1}$ of sets.
\begin{defn}~\cite[Definition 5.4.1.2]{BD}. Let $t$ be a shuffle tree monomial. For each leaf $\ell$ of $ t$ in the total order, we record the decorations of vertices of the path from the root of $ t$ to $\ell$, forming a word in the alphabet $\cal{V}$. The sequence of these words, denoted $\operatorname{Path}(t)$, is called {\bf the path sequence of the tree monomial} $t$.
\end{defn}

The following ordering is adapted from~\cite[Section 3.2.1]{DK09} to suit our purposes.

\begin{defn}\label{defn:lexpath}
Let $\mathcal{V}=\{\mathcal{V}(n)\}_{n\geq1}$ be a family of sets with a well-order on $\mathcal{V}$.
Each tree monomial determines a sequence of leaf labels and a sequence of paths.
To compare two tree monomials, we proceed as follows:
\begin{enumerate}
\item First, compare the sequences of leaf labels using reverse degree-lexicographic order (i.e., from right to left).

\item If these are equal, compare the path sequences word by word using the standard degree-lexicographic order.
\end{enumerate}
This ordering is called the {\bf leaf-lexicographic-path ordering}.
\end{defn}

We give an example for better understanding.

\begin{exam}\mlabel{ex:lexpath}
Let $\mathcal{V}=\mathcal{V}(2)=\{\mu,\nu\}$ with  $\nu>\mu$. For the shuffle tree monomials from Example~\ref{ex:st}, the path sequences are
$(\nu\mu, \nu\mu, \nu)$, $(\nu\mu, \nu, \nu\mu)$ and $(\nu, \nu\mu, \nu\mu)$,
respectively.
Moreover, we have
{\small \begin{align*}
\twool \nu\mu312\longmapsto (123, (\nu\mu, \nu\mu, \nu)),\quad \twool \nu\mu213\longmapsto (132, (\nu\mu, \nu, \nu\mu)),\quad
\twoor \nu1\mu23 \longmapsto (123, (\nu, \nu\mu, \nu\mu)).
 \end{align*}}
Since $123 > 132$ by the reverse degree lexicographic order and $\nu\mu > \nu$ by the degree-lexicographic order,
we conclude
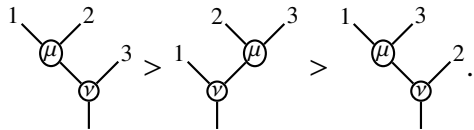
\begin{align*}
\twool \nu\mu312>\twoor \nu1\mu23>
\twool \nu\mu213.
 \end{align*}
\end{exam}

\begin{remark}
By adjusting the comparison order specified in Definition~\ref{defn:lexpath},
compare (b) first and then (a),
we arrive at the {\bf path-lexicographic ordering} defined in~\cite[Section 3.2.1]{DK09}. Under this  path-lexicographic ordering framework, the ordered result for Example~\ref{ex:lexpath} is as follows:
\begin{align*}
\twool \nu\mu312>\twool \nu\mu213>\twoor \nu1\mu23.
 \end{align*}
\end{remark}

\subsection{Gr\"obner-Shirshov bases for operads}
A systematic framework for Gr\"obner bases in the operadic context was established in~\cite{DK09, LV12}.
In this subsection, we translate it into the framework of Gr\"obner-Shirshov bases,
providing an alternative yet equivalent formulation to the Gr\"obner basis theory described in~\cite{DK09}.

To facilitate this approach, we employ the notion of $\star$-trees, which can be viewed as the tree-theoretic counterparts of $\star$-bracketed words~\cite{BCQ,GSZ}. These $\star$-trees provide a combinatorial representation of partial compositions in free operads and are essential in defining compositions, leading monomials, and reduction procedures necessary for establishing a Gr\"obner-Shirshov basis.

\begin{defn}
Let $\star=\{\star_n\}_{n\geq 1}$ be a family of symbols disjoint from the collection
$\cal{V}=\{\cal V(n)\}_{n\geq 1}$ of sets, and let $\cal V^\star=\{\cal V(n)\sqcup \{\star_n\}\}_{n\geq 1}$.
\begin{enumerate}
\item  A tree $q$ in $\cal T_{\shuffle}({\cal V^\star})$ is called a {\bf $\star$-tree} on $\cal V$ if it contains exactly one occurrence of a symbol from the family $\star=\{\star_n\}_{n\geq 1}$.
The set of all such $\star$-trees is denoted by $\cal{T}_{\shuffle}(\cal{V})^\star$.
We write $q\in \cal{T}_{\shuffle}(\cal{V})^{\star_n}$ to indicate that the unique occurrence of a $\star$-symbol in $q$ is the symbol $\star_n$.

\item  For $q\in \cal{T}_{\shuffle}(\cal{V})^{\star_n}$ and $u \in\cal{T}_{\shuffle}(\cal{V})(n)$ with $n\geq 1$, we define $q|_{ u}$ to be the tree obtained by replacing the $\star_n$ in $q$ by $u$. This can be extended linearly by $q\suba{s}:=\sum_i c_i q\suba{u_i}$ for
$s =\sum_i c_i u_i \in \bfk\cal{T}_{\shuffle}(\cal{V})(n)$.
\end{enumerate}
\mlabel{defn:startree}
\end{defn}

For example,
$$\text{if } q=\twool \mu{\star_2 }312\,\text{ and }\, u=\oneo {\nu}12, \text{ then } q|_u={\twool \mu\nu312}. 
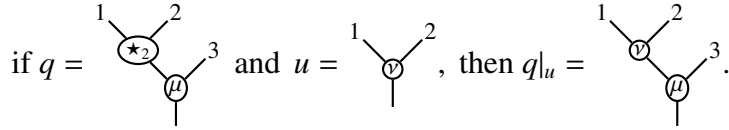
$$

\begin{defn}
A {\bf monomial order} on $\cal{T}_{\shuffle}(\cal{V})$ is a well-order $\leq$ on $\cal{T}_{\shuffle}(\cal{V})$ satisfying
$$u>v\Rightarrow q|_u > q|_v, \quad\text{ for }\, u,v\in \cal{T}_{\shuffle}(\cal{V})(n), \, q\in\cal{T}_{\shuffle}(\cal{V})^{\star_n}, \, n\geq 1.$$
\end{defn}

Let $\leq$ be a monomial order on $\cal{T}_{\shuffle}(\cal{V})$. For $f\in \bfk\cal{T}_{\shuffle}(\cal{V})$,
denote by $\bar{f}$ the leading monomial of $f$.

\begin{lemma}
 Let $\mathcal{V}=\{\mathcal{V}(n)\}_{n\geq1}$ be a family of sets with a well-order. Then the leaf-lexicographic-path order on $\cal{T}_{\shuffle}(\cal{V})$ is a monomial order.
\end{lemma}
\begin{proof}
It follows from~\cite[ Proposition 5.4.1.7]{BD}.
\end{proof}

\begin{defn}
Let $\leq$ be a monomial order on $\cal{T}_{\shuffle}(\cal{V})$ and $f, g\in \bfk\cal{T}_{\shuffle}(\cal{V})$ monic with respect to the $\leq$. Define two kinds of compositions.
\begin{enumerate}
\item If there exist $u, v, w \in
\cal{T}_{\shuffle}(\cal{V})$ such that 
$\bar{f}\circa \bar{g}:=w=\bar{f}\circ_i u=v\circ_j \bar{g}$ with $\max \big\{|\operatorname{V}(\bar{f})|, |\operatorname{V}(\bar{g})|\big\} < |\operatorname{V}(w)|<|\operatorname{V}(\bar{f})|+|\operatorname{V}(\bar{g})|$,
we call the tree polynomial
$$(f, g)_{w}:=f\circ_i u-v\circ_j g$$
the {\bf intersection composition of $f$ and $g$ with respect to} $(u,v)$.

\item  If there exist $q \in \cal{T}_{\shuffle}(\cal{V})^\star$ and $w \in \cal{T}_{\shuffle}(\cal{V})$ such  that
$\bar{f}\circa \bar{g}:=w=\bar{f}=q|_{\bar{g}},$ we call the tree polynomial
$$(f, g)_{w}:=f-q|_{g}$$
the {\bf including composition of $f$ and $g$ with respect to} $q$.
\end{enumerate}
The above tree monomial $w$ is called the {\bf ambiguity} of the composition $(f, g)_{w}$.
\mlabel{defn:comp}
\end{defn}

\begin{defn}
Let $\leq$ be a monomial order on $\cal{T}_{\shuffle}(\cal{V})$,
$S = \{S(n)\}_{n\geq 1}$ 
 a set of monic tree polynomials in $\bfk\cal{T}_{\shuffle}(\cal{V})$ and  $w \in \cal{T}_{\shuffle}(\cal{V}).$
\begin{enumerate}
\item For $u, v\in \bfk\cal{T}_{\shuffle}(\cal{V}),$ we say that $u$ and $v$ are {\bf congruent modulo} $(S, w)$ and denote this by
    $$u\equiv v\, \text{ mod }\, (S, w)$$
    if $u-v=\Sigma_{i}c_{i}q_{i}|_{s_{i}}$ for some
\[
c_{i}\in \bfk\setminus\{0\}, \quad q_{i}\in \cal{T}_{\shuffle}(\cal{V})^{\star_{n_i}}, \quad s_{i}\in S(n_i)\,\text{ such that }\, q_{i}|_{\bar{s_{i}}}<w.
\]

\item Let $f, g \in \bfk\cal{T}_{\shuffle}(\cal{V})$,
$S$ be a set of monic  polynomials in $\bfk\cal{T}_{\shuffle}(\cal{V}).$ Then the composition $(f, g)_{w}$ is  called {\bf trival modulo $(S,w)$} if
$$(f, g)_{w}\equiv 0\,\text{ mod }\,(S, w).$$
\end{enumerate}
\end{defn}
\begin{defn}
Let $\leq$ be a monomial order on $\cal{T}_{\shuffle}(\cal{V})$ and $S\subseteq \bfk\cal{T}_{\shuffle}(\cal{V})$
a set of monic polynomials. The $S$ is called a {\bf Gr\"{o}bner-Shirshov bases} with respect to $\leq$, if for all pairs $f, g \in S$, every intersection composition of the form
$(f, g)_{w}\emph{}$ is trivial modulo $(S, w)$ and every including composition of the form $(f,g)_{w}$ is trivial modulo $(S, w).$
\end{defn}

The following is the operadic version of the Composition-Diamond lemma, which serves as a fundamental tool to compute normal forms of operadic expressions and to construct explicit bases for quotient operads defined by generators and relations.

\begin{theorem}{\rm{(Composition-Diamond lemma)~\label{Composition-Diamond lemma}~\cite[Theorem 1]{DK09}
Let $\leq$ be a monomial order on $\cal{T}_{\shuffle}(\cal{V})$ and $S = \{S(n)\}_{n\geq 1}$ 
 a set of monic tree polynomials in $\bfk\cal{T}_{\shuffle}(\cal{V})$.  Then the following
statements are equivalent:
 \begin{enumerate}
\item[(I)] $S $ is a Gr\"{o}bner-Shirshov basis in $\bfk\cal{T}_{\shuffle}(\cal{V})$.

\item[(II)]  $ f\in \Id(S)\Rightarrow \bar{f}=q|_{\overline{s}}$
for some $n\geq 1, q \in \cal{T}_{\shuffle}(\cal{V})^{\star_n}$ and $s\in S(n)$.

\item[(II')]  $f\in \Id(S)\Rightarrow
f=\alpha_1q_1|_{s_1}+\alpha_2q_2|_{s_2}+ \cdots+ \alpha_k q_k|_{s_k}$
for some $k\geq 1, q_i\in \cal{T}_{\shuffle}(\cal{V})^{\star_{n_i}}$ and $s_i\in S(n_i)$ such that $q_1|_{\overline{s_1}}>q_2|_{\overline{s_2}}>\cdots>q_k|_{\overline{s_k}}$.

\item[(III)] $\bfk\cal{T}_{\shuffle}(\cal{V})=\bfk\operatorname{Irr}(S) \oplus \operatorname{Id}(S)$, where
$$
\operatorname{Irr}(S):=\cal{T}_{\shuffle}(\cal{V}) \backslash\left\{\left.q\right|_{\bar{s}} \mid q \in \cal{T}_{\shuffle}(\cal{V})^{\star_n}, s \in S(n),  n\geq 1 \right\},
$$
and $\operatorname{Irr}(S)$ is a \bfk-basis of \bfk $\cal{T}_{\shuffle}(\cal{V}) / \operatorname{Id}(S)$.
\end{enumerate}}}
\mlabel{thm:cdl}
\end{theorem}

\begin{remark}
The set $S$ forms a Gr\"{o}bner-Shirshov basis in $\bfk\cal{T}_{\shuffle}(\cal{V})$ if and only if the associated rewriting system~\cite{BD}
$$\Pi_S:=\{q|_{\bar{s}} \to q|_{\bar{s}-s}~|~s\in S(n), q\in\cal{T}_{\shuffle}(\cal{V})^{\star_n} \}$$
is convergent. In practice, the method of rewriting systems is commonly used to prove that the set $S$ is a Gr\"{o}bner-Shirshov basis in $\bfk\cal{T}_{\shuffle}(\cal{V})$~\cite{BD,Dot}.
Specifically, it suffices to show that the rewriting system $\Pi_S$ is confluent, given that  $\Pi_S$ is already terminating. This termination property follows directly from the monomial order employed in the construction, which is a well-order. Throughout the remainder of the paper, we work with the monomial order introduced in Definition~\mref{defn:lexpath}.
\mlabel{lem:rs}
\end{remark}

\subsection{Disuccessors of binary operads}
In this subsection, we review the notion of the disuccessor  for both nonsymmetric and symmetric operads. This concept plays a central role in the operadic approach to splitting algebraic operations and generalizing classical algebraic structures through operadic transformations.

\begin{defn}~\cite{BBG13}
Let $\gensp$ be a set. Denote
\begin{equation}
\widetilde{\gensp} := \gensp \times \{ \prec ,  \succ\}.
\mlabel{eq:tsp}
\end{equation}
 For each $\gop\in \cal{V}$, denote $\prec_\gop:= (\gop , \prec)\,\text{ and }\, \succ_\omega:=(\gop , \succ).$
\end{defn}

The formal definition of the disuccessor operation $\su_x(t)$ is provided in~\cite[Definition~2.3]{BBG13}. However, in many cases, it is more convenient to work with the following equivalent, yet more accessible, formulation.

\begin{prop}~\cite[Proposition 2.4]{BBG13}{\rm{
Let $\gensp$ be a set, $ t\in \calt(\gensp)$ and $x\in\lin( t)$. The \suc $\su_x( t)$ of $ t$ is obtained by redecorating each vertex of $ t$ according to the following rules:
\begin{enumerate}
\item We replace the decoration $\gop$ of each vertex on the path from the root to the leave $x$ of $ t$ by
\begin{enumerate}
\item[(i)] $\prec_\gop$ if the path turns left at this vertex.
\item[(ii)] $\succ_\gop$ if the path turns right at this vertex.
\end{enumerate}
\item We replace the decoration $\gop$ of each vertex not on the path from the root to the leave $x$ of $ t$ by $\ast_\gop:=\prec_\gop +\succ_\gop$.
\end{enumerate}}}
\mlabel{succpath}
\end{prop}

Let us expose an example heuristically.

\begin{exam} We have
\begin{align*}
\su_2\left(\threeoY {\nu}{\mu}{\omega}1234\right)=&\threeoY {\prec_{\nu}}{\succ_{\mu}}{\ast_{\omega}}1234=
\threeoY {\prec_{\nu}}{\succ_{\mu}}{\prec_{\omega}}1234+\threeoY {\prec_{\nu}}{\succ_{\mu}}{\succ_{\omega}}1234.
\end{align*}
\end{exam}

We now recall the construction of the disuccessor for symmetric operads.
Let $\opd=\bfk\mathcal{T}(\gensp)/ (\relsp)$  be a binary symmetric (
{parallel to the case of symmetric operads, the following definition can be performed in the nonsymmetric framework. }) operad,
 \[ \gensp  =\Big\{\oneo \mu12,\,\oneo \nu12\Big\},\quad\text{ where }\, \mu^{(12)}=\nu, ~~~~(12)\in \BS_{2}\]
such that $\relsp$ is spanned, as an $\BS$-module, by elements of the form
\begin{equation*}
\relsp:=\left\{r_s=\sum_i c_{s,i} t_{s,i}\in\bfk\mathcal{T}(\gensp)(n) \;\Big| \;\ c_{s,i}\in\bfk , \ 1\leq s\leq k\right\}.
\end{equation*}

\begin{defn}~\cite[Definition 2.19]{BBG13}
The {\bf \suc} of $\opd$ is defined to be the binary operad
$$\su(\opd):=\bfk\mathcal{T}( \widetilde{\gensp}  )/ (\su(\relsp)),$$
where $\widetilde{\gensp}$ is given in~(\ref{eq:tsp}) and
\begin{equation*}
\su(R):= \left\lbrace  \su_x(r_{s}):=\sum_ic_{s,i}\su_{x}(t_{s,i})\;\Big| \; x\in \lin(t_{s,i}), \; 1\leq s\leq
k  \right\rbrace . 
\end{equation*}
Here the action of $S_2$ on $\widetilde{\gensp}$ is given by: $\prec_\mu^{(12)}=\succ_\nu$ and $\succ_\mu^{(12)}=\prec_\nu$.
\mlabel{defn:suc}
\end{defn}

For $N\geq 2$, the {\bf $N$-th \suc} of $\mathcal{P}$, denoted by $\su^{N}(\np)$, is defined recursively as the \suc of
the {\bf $(N-1)$-th \suc} of the operad, where the {\bf first \suc} of the operad is just the \suc of the operad.

\section{Disuccessors of the operads $\as, \lie$ and $\prelie$, and their Gr\"obner-Shirshov-bases}\label{sec:bsoperad}
In this section, we establish that the disuccessors of the associative operad, the Lie operad, and the pre-Lie operad each admit a Gr\"obner-Shirshov basis. These results demonstrate the compatibility of the disuccessor construction with the theory of Gr\"obner-Shirshov bases and highlight its structural preservation across fundamental operadic types.

\subsection{Disuccessor of the operad $\as$ and Gr\"obner-Shirshov bases}
\label{sec:As operad}
We begin by using the associative operad $\as$ as a representative example to illustrate the disuccessor construction for binary operads.

Let $\as $  be the associative (nonsymmetric) operad with the operation $\cal{V}_\as=\{\bcdot\}$  and the relation
\[R_\as =\twool \bcdot\bcdot312-\twoor \bcdot1\bcdot23.\]
Then $\su(\as )$ is the operad $\dend$ of dendriform algebras~\cite{BBG13}. 
For elements of $\widetilde{\cal{V}_\as}$, we abbreviate $\prec_\bcdot$ as $\prec$,  $\succ_\bcdot$ as $\succ$.

\begin{prop}~\label{prop:as}
Equipped with the leaf-lexicographic-path order on $\fnsop(\widetilde{\cal{V}_\as})$ such that $\succ$ is greater than $\prec$ in $ \widetilde{\cal{V}_\as}$,  we obtain
\begin{equation*}
\lbar{\su_{i}(R_\as)}=\lbar{\su_{i}(\lbar{R_\as })}, \quad\text{ for }\, i=1,2,3. \mlabel{eq:u1}
\end{equation*}
Moreover, for any ambiguity $\lbar{\su_{i_1}(R_\as)} \circa \lbar{\su_{i_2}(R_\as)}$ arising from the relations in $\su(R_\as)$, as defined in Definition~\mref{defn:comp}, with $i_1, i_2 \in \{1,2,3\}$, there exists $i \in \{1,2,3,4\}$ such that
$$\lbar{\mathop{\su_{i_1}}(R_\as )}\circa\lbar{\mathop{\su_{i_2}}(R_\as )}=\lbar{\mathop{\su_i}(\lbar{R_\as} \circa \lbar{R_\as} )}.$$
\end{prop}

\begin{proof}
By the definition of disuccessor operation, we have
\begin{align*}
  \su_1(R_\as )=&\llbar{\twool {\prec}{\prec}312}-\twoor {\prec}1{\ast}23,\quad
  \su_2(R_\as )=\llbar{\twool {\prec}{\succ}312}-\twoor {\succ}1{\prec}23,\\
  \su_3(R_\as )=&\twool {\succ}{\ast}312-\twoor {\succ}1{\succ}23=\llbar{\twool {\succ}{\succ}312}+\twool {\succ}{\prec}312-\twoor {\succ}1{\succ}23.
\end{align*}
The ambiguity of ${R_\as }$ is
 \[ \lbar{R_\as} \circa\lbar{R_\as} =\threeoll \bcdot\bcdot4\bcdot312,\]
where the left hand side is defined in Definition~\mref{defn:comp}. Note that $\succ$ is bigger than $\prec$ by the hypothesis. The ambiguities of
  \[\su(R_\as )=\Big\{\su_1(R_\as ),\su_2(R_\as ),\su_3(R_\as )\Big\}\]
  are as follows:
\begin{equation}
\begin{aligned}
\lbar{\su_1(R_\as )}\circa\lbar{\su_1(R_\as )}=&\threeoll {\prec}{\prec}4{\prec}312,\quad
\lbar{\su_1(R_\as )}\circa\lbar{\su_2(R_\as )}=\threeoll {\prec}{\prec}4{\succ}312,\\
\lbar{\su_2(R_\as )}\circa\lbar{\su_3(R_\as )}=&\threeoll {\prec}{\succ}4{\succ}312,\quad
\lbar{\su_3(R_\as )}\circa\lbar{\su_3(R_\as )}=\threeoll {\succ}{\succ}4{\succ}312.
\end{aligned}
\label{eq:ambasso}
\end{equation}
Then
\begin{align*}
&\lbar{\su_1(\lbar{R_\as} \circa \lbar{R_\as} )}
=\lbar{\su_1(R_\as )}\circa\lbar{\su_1(R_\as )},\\
&\lbar{\su_2(\lbar{R_\as} \circa \lbar{R_\as} )}
=\lbar{\su_1(R_\as )}\circa\lbar{\su_2(R_\as )},\\
&\lbar{\su_3(\lbar{R_\as} \circa \lbar{R_\as} )}
=\lbar{\su_2(R_\as )}\circa\lbar{\su_3(R_\as )},\\
&\lbar{\su_4(\lbar{R_\as} \circa \lbar{R_\as} )}
=\lbar{\su_3(R_\as )}\circa\lbar{\su_3(R_\as )},
\end{align*}
as required.
\end{proof}

The set of relations of $\dend$ operad is a Gr\"obner-Shirshov basis as stated in~\cite{Mada14}. Here we present a concise proof for this assertion.

\begin{theorem}\mlabel{thm:gsbdend}
Equipped with the leaf-lexicographic-path order on $\fnsop(\widetilde{\cal{V}_\as})$ such that $\succ$ is greater than $\prec$ in $\widetilde{\cal{V}_\as}$,
the set of relations $\su(R_\as) = R_{\dend}$ defining the operad $\dend$ forms a Gr\"obner-Shirshov basis in the free nonsymmetric operad
$\bfk\fnsop(\widetilde{\cal{V}_\as})$.
\end{theorem}

\begin{proof}
Since $\su(R_\as) = \dend$~\cite{BBG13}, we have $\su(R_\as)=R_{\dend}.$
According to Remark~\ref{lem:rs}, we employ the method of  rewriting systems to prove that $\su(R_\as)$ is a Gr\"obner-Shirshov basis.
There are four ambiguities  (equivalently, local forks) listed in~(\mref{eq:ambasso}).
The confluence of the rewriting system is verified through direct computation.
Two such computations are detailed in Figure~\ref{fig:didend}:
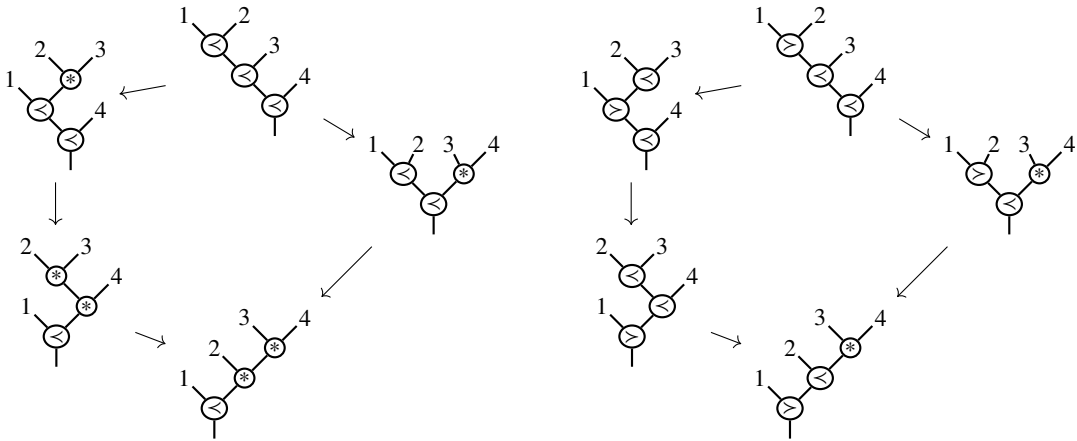
\begin{figure}[H]
\centering
\begin{tikzpicture}
\node[draw=none,fill=none] at (0,0) (ass1) {\threeoll {\prec}{\prec}4{\prec}312};
\node[draw=none,fill=none] at (-2.5,-0.45) (ass2) { \threeolr {\prec}{\prec}41{\ast}23};
\node[draw=none,fill=none] at (-2.5,-3.05) (ass3) {\quad\threeorl {\prec}1{\ast}{\ast}423};
\node[draw=none,fill=none] at (2.5,-1.5) (ass4) {\threeoY {\prec}{\prec}{\ast}1234};
\node[draw=none,fill=none] at (0,-4) (ass5) {\threeorr {\prec}1{\ast}2{\ast}34};
\draw [->] (ass1) -- (ass2);
\draw [->] (ass2) -- (ass3);
\draw [->] (ass3) -- (ass5);
\draw [->] (ass1) -- (ass4);
\draw [->] (ass4) -- (ass5);
\end{tikzpicture}
\quad
\begin{tikzpicture}
\node[draw=none,fill=none] at (0,0) (ass1) {\threeoll {\prec}{\prec}4{\succ}312};
\node[draw=none,fill=none] at (-2.5,-0.45) (ass2) { \threeolr {\prec}{\succ}41{\prec}23};
\node[draw=none,fill=none] at (-2.5,-3.05) (ass3) {\quad\threeorl {\succ}1{\prec}{\prec}423};
\node[draw=none,fill=none] at (2.5,-1.5) (ass4) {\threeoY {\prec}{\succ}{\ast}1234};
\node[draw=none,fill=none] at (0,-4) (ass5) {\threeorr {\succ}1{\prec}2{\ast}34};
\draw [->] (ass1) -- (ass2);
\draw [->] (ass2) -- (ass3);
\draw [->] (ass3) -- (ass5);
\draw [->] (ass1) -- (ass4);
\draw [->] (ass4) -- (ass5);
\end{tikzpicture}
\caption{Two diamonds for the operad $\su(\as )=\dend$}
\label{fig:didend}
\end{figure}
while the others follow analogously.
\end{proof}

\subsection{Disuccessor of the operad $\lie$ and Gr\"obner-Shirshov bases}
\label{sec:lie operad}
In this subsection, we turn our attention to the Lie operad $\lie$ as a case study to illustrate the construction of successors for binary operads. By applying this framework, we derive a Gr\"obner-Shirshov basis for the pre-Lie operad $\prelie$, thereby establishing a concrete link between the operadic disuccessor process and the combinatorial structure of pre-Lie algebras.

\begin{defn}
{\bf A (right) pre-Lie algebra} is defined by one bilinear operation $\bullet$ and one relation:
$$ (x\bullet y)\bullet z-x\bullet(y\bullet z)=(x\bullet z)\bullet y-x\bullet(z\bullet y) \ .$$
The associated operad is denoted by $\prelie$.
\end{defn}

It is known that $\su(\lie)=\prelie$~\cite{BBG13}, where $\lie $  is the Lie operad with the operation $\cal{V}_\lie=\{\mu=-\mu^{(12)}\}$  and relation $$R_\lie=\twool {\mu}{\mu}312-\twoor {\mu}1{\mu}23-\twool {\mu}{\mu}213.$$
Then $\widetilde{\cal{V}_\lie}=\{\rp_\mu, \lp_\mu\}$. Moreover, by $\mu=-\mu^{(12)}$, we have $\prec_\mu=-\succ_\mu^{(12)}$. 
\begin{prop}~\label{prop:lie}
Equipped with the leaf-lexicographic-path order on $\cal{T}_{\shuffle}(\widetilde{\cal{V}_\lie})$ such that $\succ_\mu$ is greater than $\prec_\mu$ in $\widetilde{\cal{V}_\lie}$,
then
\begin{equation*}
\lbar{\su_{i}(R_\lie)}=\lbar{\su_{i}(\lbar{R_\lie })}, \quad\text{ for }\, i=1,2,3. \mlabel{eq:lie}
\end{equation*}
Moreover, for any ambiguity $\lbar{\mathop{\su_{i_1}}(R_\lie )}\circa\lbar{\mathop{\su_{i_2}}(R_\lie )}$ from the relation $\su{(R_\lie)}$ with some $i_1, i_2\in \{1,2,3\}$, there exists $i\in \{1,2,3,4\}$ such that
$$\lbar{\mathop{\su_{i_1}}(R_\lie )}\circa\lbar{\mathop{\su_{i_2}}(R_\lie )}=\lbar{\mathop{\su_i}(\lbar{R_\lie} \circa \lbar{R_\lie} )}.$$
\end{prop}

\begin{proof}
For the first statement, notice that 
\begin{align*}
\su_1(R_\lie)=&\llbar{\twool {\prec_\mu}{\prec_\mu}312}-\twoor {\prec_\mu}1{\prec_\mu}23-\twoor {\prec_\mu}1{\succ_\mu}23-\twool {\prec_\mu}{\prec_\mu}213,\\
%
\su_2(R_\lie)=&\llbar{\twool {\prec_\mu}{\succ_\mu}312}-\twoor {\succ_\mu}1{\prec_\mu}23-\twool {\succ_\mu}{\succ_\mu}213-\twool {\succ_\mu}{\prec_\mu}213,\\
\su_3(R_\lie)=&\llbar{\twool {\succ_\mu}{\succ_\mu}312}+\twool {\succ_\mu}{\prec_\mu}312-\twoor {\succ_\mu}1{\succ_\mu}23-\twool {\prec_\mu}{\succ_\mu}213.\\
\end{align*}
Thus
\begin{align*}
\lbar{\su_{1}(R_\lie)}=& \llbar{\twool {\prec_\mu}{\prec_\mu}312}=\lbar{\su_{1}(\lbar{R_\lie })},~~~~\quad
\lbar{\su_{2}(R_\lie)}=\llbar{\twool {\prec_\mu}{\succ_\mu}312}=\lbar{\su_{2}(\lbar{R_\lie })},\\
\lbar{\su_{3}(R_\lie)}=& \llbar{\twool {\succ_\mu}{\succ_\mu}312} = \lbar{\su_{3}(\lbar{R_\lie })}.
\end{align*}

For the second statement, the ambiguity of $R_\lie$ is
\[\lbar{R_\lie}\circa\lbar{R_\lie}=\threeoll \mu\mu4\mu312.\]
By the hypothesis $\succ_\mu>\prec_\mu$, according to  leaf-lexicographic-path order, the ambiguities of $\su(R_\lie)$ are
\begin{equation}\label{eq:ovepll}
\begin{split}
\lbar{\su_1(R_\lie)}\circa \lbar{\su_1(R_\lie)}=\threeoll {\prec_\mu}{\prec_\mu}4{\prec_\mu}312,\quad \lbar{\su_1(R_\lie)}\circa\lbar{\su_2(R_\lie)}=\threeoll {\prec_\mu}{\prec_\mu}4{\succ_\mu}312,\\
\lbar{\su_2(R_\lie)}\circa\lbar{\su_3(R_\lie)}=\threeoll {\prec_\mu}{\succ_\mu}4{\succ_\mu}312,\quad \lbar{\su_3(R_\lie)}\circa\lbar{\su_3(R_\lie)}=\threeoll {\succ_\mu}{\succ_\mu}4{\succ_\mu}312.
\end{split}
\end{equation}
Thus
\begin{align*}
&\lbar{\su_1(\lbar{R_\lie} \circa \lbar{R_\lie} )}
=\lbar{\su_1(R_\lie )}\circa\lbar{\su_1(R_\lie )},\\
&\lbar{\su_2(\lbar{R_\lie} \circa \lbar{R_\lie} )}
=\lbar{\su_1(R_\lie )}\circa\lbar{\su_2(R_\lie )},\\
&\lbar{\su_3(\lbar{R_\lie} \circa \lbar{R_\lie} )}
=\lbar{\su_2(R_\lie )}\circa\lbar{\su_3(R_\lie )},\\
&\lbar{\su_4(\lbar{R_\lie} \circa \lbar{R_\lie} )}
=\lbar{\su_3(R_\lie )}\circa\lbar{\su_3(R_\lie )},
\end{align*}
as required.
\end{proof}

We are ready for the main result in this subsection.
The set of relations of $\prelie$ operad is a Gr\"obner-Shirshov basis as stated in~\cite{BD}. Here we present a concise proof for this assertion.

\begin{theorem}\label{thm:lipre}
Equipped with the leaf-lexicographic-path order on $\cal{T}_{\shuffle}(\widetilde{\cal{V}_\lie})$ such that $\succ_\mu$ is greater than $\prec_\mu$ in $\widetilde{\cal{V}_\lie}$, the set of relations $\su(R_\lie) = R_{\prelie}$ of $\prelie$ operad is a Gr\"obner-Shirshov basis in $\bfk\cal T_{\shuffle}(\widetilde{\cal{V}_\lie})$.
\end{theorem}

\begin{proof}
Using Remark~\ref{lem:rs}, we employ the method of  rewriting systems to prove that $\su(R_\lie)$ is a Gr\"obner-Shirshov basis.
There are four ambiguities (equivalently, local forks) listed in~\eqref{eq:ovepll}.
We claim that the diamond for the operad $\su(\lie )=\prelie$ is the disuccessor of Figure~\ref{fig:dilie}.
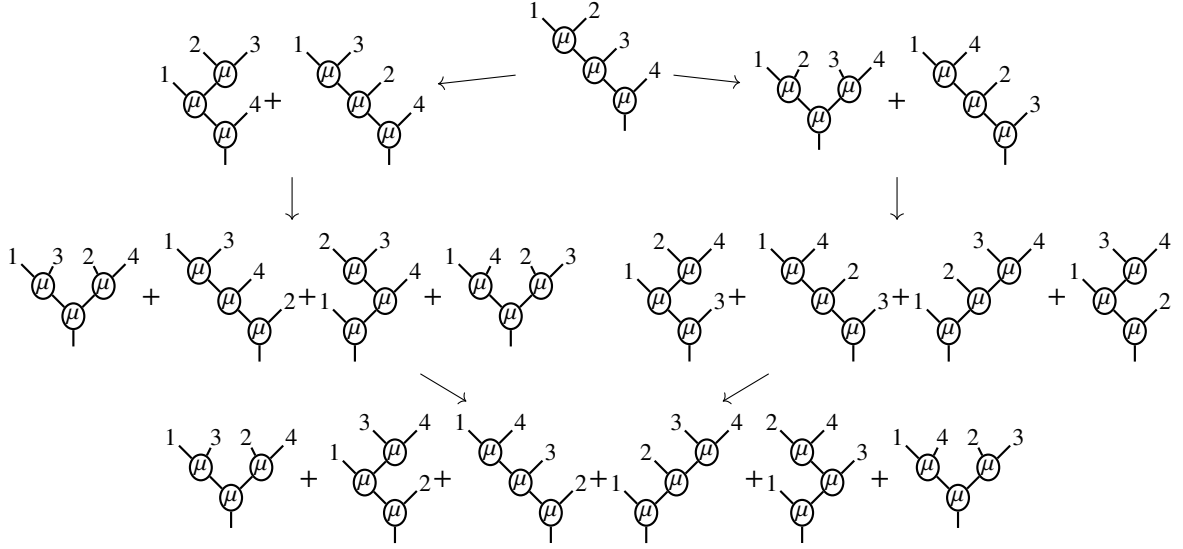
\begin{figure}[h]
\centering
\begin{tikzpicture}
\node[draw=none,fill=none] at (0,0) (lie1) {\threeoll {\mu}{\mu}4{\mu}312};
\node[draw=none,fill=none] at (-4,-0.45) (lie2) {\threeolr {\mu}{\mu}41{\mu}23+ \threeoll {\mu}{\mu}4{\mu}213};
\node[draw=none,fill=none] at (-4,-3.05) (lie3) {\threeoY {\mu}{\mu}{\mu}1324+\threeoll {\mu}{\mu}2{\mu}413+\threeorl {\mu}1{\mu}{\mu}423+\threeoY {\mu}{\mu}{\mu}1423};
\node[draw=none,fill=none] at (4,-0.45) (lie4) {\threeoY {\mu}{\mu}{\mu}1234+\threeoll {\mu}{\mu}3{\mu}214};
\node[draw=none,fill=none] at (4,-3.05) (lie5) {\threeolr {\mu}{\mu}31{\mu}24+ \threeoll {\mu}{\mu}3{\mu}214+\threeorr {\mu}1{\mu}2{\mu}34+\threeolr {\mu}{\mu}21{\mu}34};
\node[draw=none,fill=none] at (0,-5.45) (lie6) {
\threeoY {\mu}{\mu}{\mu}1324+ \threeolr {\mu}{\mu}21{\mu}34+\threeoll {\mu}{\mu}2{\mu}314+\threeorr {\mu}1{\mu}2{\mu}34+\threeorl {\mu}1{\mu}{\mu}324+\threeoY {\mu}{\mu}{\mu}1423};
\draw [->] (lie1) -- (lie2);
\draw [->] (lie2) -- (lie3);
\draw [->] (lie3) -- (lie6);
\draw [->] (lie1) -- (lie4);
\draw [->] (lie4) -- (lie5);
\draw [->] (lie5) -- (lie6);
\end{tikzpicture}
\caption{The diamond for the operad $\lie$}
\label{fig:dilie}
\end{figure}
We just check one diamond of them, see Figure~\ref{fig:diplie}, as the others are similar.
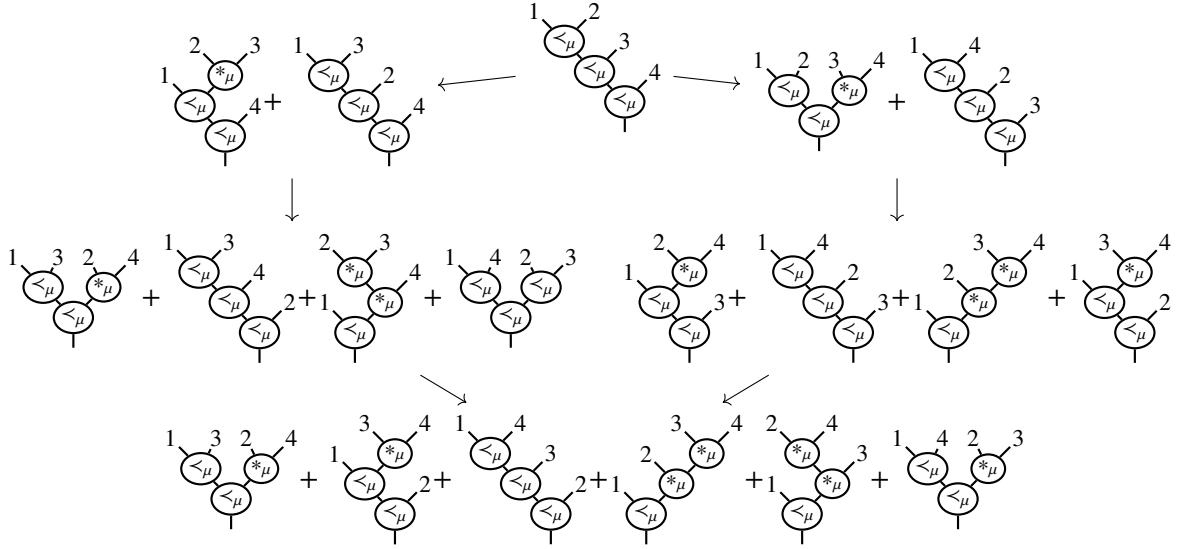
\begin{figure}[H]
\centering
\begin{tikzpicture}
\node[draw=none,fill=none] at (0,0) (lie1) {\threeoll {\prec_\mu}{\prec_\mu}4{\prec_\mu}312};
\node[draw=none,fill=none] at (-4,-0.45) (lie2) {\threeolr {\prec_\mu}{\prec_\mu}41{\ast_\mu}23+ \threeoll {\prec_\mu}{\prec_\mu}4{\prec_\mu}213};
\node[draw=none,fill=none] at (-4,-3.05) (lie3) {\threeoY {\prec_\mu}{\prec_\mu}{\ast_\mu}1324+\threeoll {\prec_\mu}{\prec_\mu}2{\prec_\mu}413+\threeorl {\prec_\mu}1{\ast_\mu}{\ast_\mu}423+\threeoY {\prec_\mu}{\prec_\mu}{\prec_\mu}1423};
\node[draw=none,fill=none] at (4,-0.45) (lie4) {\threeoY {\prec_\mu}{\prec_\mu}{\ast_\mu}1234+\threeoll {\prec_\mu}{\prec_\mu}3{\prec_\mu}214};
\node[draw=none,fill=none] at (4,-3.05) (lie5) {\threeolr {\prec_\mu}{\prec_\mu}31{\ast_\mu}24+ \threeoll {\prec_\mu}{\prec_\mu}3{\prec_\mu}214+\threeorr {\prec_\mu}1{\ast_\mu}2{\ast_\mu}34+\threeolr {\prec_\mu}{\prec_\mu}21{\ast_\mu}34};
\node[draw=none,fill=none] at (0,-5.45) (lie6) {
\threeoY {\prec_\mu}{\prec_\mu}{\ast_\mu}1324+ \threeolr {\prec_\mu}{\prec_\mu}21{\ast_\mu}34+\threeoll {\prec_\mu}{\prec_\mu}2{\prec_\mu}314+\threeorr {\prec_\mu}1{\ast_\mu}2{\ast_\mu}34+\threeorl {\prec_\mu}1{\ast_\mu}{\ast_\mu}324+\threeoY {\prec_\mu}{\prec_\mu}{\ast_\mu}1423};
\draw [->] (lie1) -- (lie2);
\draw [->] (lie2) -- (lie3);
\draw [->] (lie3) -- (lie6);
\draw [->] (lie1) -- (lie4);
\draw [->] (lie4) -- (lie5);
\draw [->] (lie5) -- (lie6);
\end{tikzpicture}
\caption{One diamond for the operad $\prelie$}
\label{fig:diplie}
\end{figure}
Therefore the set $\su(R_\lie)$ is a Gr\"obner-Shirshov basis.
\end{proof}

\subsection{Disuccessor of the operad $\prelie$ and Gr\"obner-Shirshov bases}
\mlabel{sec:prel operad}
Continuing in the same vein, we now focus on the pre-Lie operad $\prelie$ as a further example to explore the disuccessor construction for binary operads. This leads naturally to the study of L-dendriform algebras, which arise as algebraic structures associated with the disuccessor of the $\prelie$ operad.

We begin by recalling the definition and key properties of L-dendriform algebras, which will serve as the foundation for our subsequent operadic analysis.

Building on the notion of left L-dendriform algebras~\cite{BLN10,NB14}, we introduce the concept of right L-dendriform algebras by considering the opposite operations.

\begin{defn}
 A {\bf (right) L-dendriform algebra} $(A, \prec, \succ)$ is a vector space $A$ equipped with two bilinear operations $\prec, \succ: A \otimes A \rightarrow A$ such that for any $x, y, z \in A$,
$$
\begin{gathered}
(x \succ y) \prec z+x \succ(z \succ y)=x \succ(y \prec  z)+(x \ast z) \succ y,\\
(x \prec y) \prec z+ x\prec(z\ast  y)=x \prec(y \ast z)+(x \prec z) \prec y,
\end{gathered}
$$
where $x \ast y :=x \prec y+x \succ y$.
\end{defn}

The symmetric operad $\prelie$ of right pre-Lie algebras, which is equal to the disucessor $\su(\lie)$ of the Lie operad $\lie$, equips with two binary operations $\bullet$, $\bullet':=\bullet^{(12)}$ and three relations
\begin{align}
\R_\prelie:=&\llbar{\twool {\bullet}{\bullet}312}-\twoor {\bullet}1{\bullet}23+\twoor {\bullet}1{\bullet'}23-\twool {\bullet}{\bullet}213
=\su_1(R_\lie)|_{\prec_\mu\mapsto \bullet,\, \succ_\mu\mapsto-\bullet' },\nonumber\\
S_\prelie:=&\llbar{\twool {\bullet}{\bullet'}312}-\twoor {\bullet'}1{\bullet}23+\twool {\bullet'}{\bullet'}213-\twool {\bullet'}{\bullet}213
=\su_2(R_\lie)|_{\prec_\mu\mapsto \bullet,\, \succ_\mu\mapsto-\bullet'} ,\label{eq:prelie}\\
T_\prelie:=&\llbar{\twool {\bullet'}{\bullet'}312}-\twool {\bullet'}{\bullet}312-\twoor {\bullet'}1{\bullet'}23+\twool {\bullet}{\bullet'}213
=\su_3(R_\lie)|_{\prec_\mu\mapsto \bullet,\, \succ_\mu\mapsto-\bullet'}.\nonumber
\end{align}
Replacing the operation $\prec_\mu$  (resp. $\succ_\mu$)  by $\bullet$ (resp. $-\bullet^{(12)}$), these relations are
equal to $\su(R_\lie)$.

Assume $$\oneo {\bullet'}12 > \oneo \bullet12.$$
Using the leaf-lexicographic-path order in Definition~\mref{defn:lexpath}, the leading terms are
\begin{equation}\label{eq:leadingterms}
\begin{split}
\lbar{\R_\prelie}&=\twool {\bullet}{\bullet}312=\lbar{\su_1(R_\lie)|_{\prec_\mu\mapsto \bullet,\, \succ_\mu\mapsto-\bullet' }},\\
\quad\lbar{S_\prelie}&=\twool {\bullet}{\bullet'}312=\lbar{\su_2(R_\lie)|_{\prec_\mu\mapsto \bullet,\, \succ_\mu\mapsto-\bullet' }}, \\
\quad\lbar{T_\prelie}&=\twool {\bullet'}{\bullet'}312=\lbar{\su_3(R_\lie)|_{\prec_\mu\mapsto \bullet,\, \succ_\mu\mapsto-\bullet' }}.
\end{split}
\end{equation}
These three monomials of~\eqref{eq:leadingterms} have four ambiguities
\begin{equation}\label{eq:ovepl}
\begin{split}
T:=\lbar{\R_\prelie}\circa \lbar{\R_\prelie}=\threeoll {\bullet}{\bullet}4{\bullet}312,\quad U:=\lbar{\R_\prelie}\circa\lbar{S_\prelie}=\threeoll {\bullet}{\bullet}4{\bullet'}312,\\
V:=\lbar{S_\prelie}\circa\lbar{T_\prelie}=\threeoll {\bullet}{\bullet'}4{\bullet'}312,\quad W:=\lbar{T_\prelie}\circa\lbar{T_\prelie}=\threeoll {\bullet'}{\bullet'}4{\bullet'}312.
\end{split}
\end{equation}

\begin{lemma}~\cite[Proposition 6.11]{BBG13}\label{prop:preLD}
The operad $\LDend$ of L-dendriform algebras is the disuccessor of $\prelie$, i.e.,
\[
\su(\prelie)=\LDend.\]
\end{lemma}

Notice that the set of operations of $\prelie$ operad is  $\cal V_{\prelie}=\{\bullet, \bullet':=\bullet^{(12)}\}$. 
By Definition~\ref{defn:suc}, we have $\widetilde{\cal V_{\prelie}}=\{\succ_\bullet, \succ_{\bullet'}, \prec_\bullet, \prec_{\bullet'}\}$ and
$$\succ_\bullet^{(12)}=\prec_{\bullet'}, \prec_\bullet^{(12)}=\succ_{\bullet'}.$$
In order to simplify the symbols, we set
\begin{equation}\label{eq:ldop}
\succ:=\succ_\bullet, \quad\succ':=\succ_\bullet^{(12)}=\prec_{\bullet'},
\quad \prec:=\prec_\bullet, \quad \prec':=\prec_\bullet^{(12)}=\succ_{\bullet'},\quad\ast:=\prec+\succ,\quad\ast':=\prec'+\succ'.
\end{equation}

Similar to Propositions~\ref{prop:as} and~\mref{prop:lie}, we have the following result.
\begin{prop}
Assume that $\cal{T}_{\shuffle}(\widetilde{\cal V_{\prelie}})$  is equipped with the leaf-lexicographic-path order in which the elements of $\cal V$ are ordered from largest to smallest as $\succ', \,\prec',\,\succ,\,\prec $. Then
 \begin{equation*}
\lbar{\su_{i}(R)}=\lbar{\su_{i}(\lbar{R })},\quad \text{ for }\, i=1,2,3\,\text{ and }\, R\in \{\R_\prelie,S_\prelie,T_\prelie\},
\end{equation*}
and the following equations hold,
\allowdisplaybreaks{{\small{
\begin{align}
\lbar{\su_1(T)}=&\lbar{\su_1(\lbar{\R_\prelie})}\circa\lbar{\su_1(\lbar{\R_\prelie})},
\quad
\lbar{\su_2(T)}=\lbar{\su_1(\lbar{\R_\prelie})}\circa\lbar{\su_2(\lbar{\R_\prelie})},\nonumber\\
\lbar{\su_3(T)}=&\lbar{\su_2(\lbar{\R_\prelie})}\circa\lbar{\su_3(\lbar{\R_\prelie})},
\quad
\lbar{\su_4(T)}=\lbar{\su_3(\lbar{\R_\prelie})}\circa\lbar{\su_1(\lbar{\R_\prelie})},\nonumber\\
\lbar{\su_1(U)}=&\lbar{\su_1(\lbar{\R_\prelie})}\circa\lbar{\su_1(\lbar{S_\prelie})},
\quad
\lbar{\su_2(U)}=\lbar{\su_1(\lbar{\R_\prelie})}\circa\lbar{\su_2(\lbar{S_\prelie})},\nonumber\\
\lbar{\su_3(U)}=&\lbar{\su_2(\lbar{\R_\prelie})}\circa\lbar{\su_3(\lbar{S_\prelie})},
\quad
\lbar{\su_4(U)}=\lbar{\su_3(\lbar{\R_\prelie})}\circa\lbar{\su_2(\lbar{S_\prelie})},\nonumber\\
\lbar{\su_1(V)}=&\lbar{\su_1(\lbar{S_\prelie})}\circa\lbar{\su_1(\lbar{T_\prelie})},
\quad
\lbar{\su_2(V)}=\lbar{\su_1(\lbar{S_\prelie})}\circa\lbar{\su_2(\lbar{T_\prelie})},\label{eq:ldendo}\\
\lbar{\su_3(V)}=&\lbar{\su_2(\lbar{S_\prelie})}\circa\lbar{\su_3(\lbar{T_\prelie})},
\quad
\lbar{\su_4(V)}=\lbar{\su_3(\lbar{S_\prelie})}\circa\lbar{\su_3(\lbar{T_\prelie})},\nonumber\\
\lbar{\su_1(W)}=&\lbar{\su_1(\lbar{T_\prelie})}\circa\lbar{\su_1(\lbar{T_\prelie})},
\quad
\lbar{\su_2(W)}=\lbar{\su_1(\lbar{T_\prelie})}\circa\lbar{\su_2(\lbar{T_\prelie})},\nonumber\\
\lbar{\su_3(W)}=&\lbar{\su_2(\lbar{T_\prelie})}\circa\lbar{\su_3(\lbar{T_\prelie})},
\quad
\lbar{\su_4(W)}=\lbar{\su_3(\lbar{T_\prelie})}\circa\lbar{\su_1(\lbar{T_\prelie})},\nonumber
\end{align}}}}
where $T, U, V$ and $W$ are in~\eqref{eq:ovepl}.
\end{prop}

\begin{proof}
Applying Lemma~\ref{prop:preLD} and notations in~\eqref{eq:ldop},
we have
{\small{\begin{align}
\su_1(\R_\prelie)=&\twool {\prec}{\prec}312
-\twoor {\prec}1{\ast}23
-\twool {\prec}{\prec}213
+\twoor {\prec}1{\ast'}23,\nonumber\\
\su_2(\R_\prelie)=&\twool {\prec}{\succ}312
-\twoor {\succ}1{\prec}23
-\twool {\succ}{\ast}213
+\twoor {\succ}1{\prec'}23,\nonumber\\
\su_3(\R_\prelie)=&\twool {\succ}{\ast}312
-\twoor {\succ}1{\succ}23
-\twool {\prec}{\succ}213
+\twoor {\succ}1{\succ'}23,\nonumber\\
\su_1(S_\prelie)=&\twool {\prec}{\prec'}312
-\twoor {\prec'}1{\ast}23
-\twool {\prec'}{\prec }213
+\twool {\prec'}{\prec'}213,\nonumber\\
\su_2(S_\prelie)=&\twool {\prec}{\succ'}312
-\twoor {\succ'}1{\prec}23
-\twool {\succ'}{\ast}213
+\twool {\succ'}{\ast'}213,\label{eq:oldend}\\
\su_3(S_\prelie)=&\twool {\succ}{\ast'}312
-\twoor {\succ'}1{\succ}23
-\twool {\prec'}{\succ}213
+\twool {\prec'}{\succ'}213,\nonumber\\
\su_1(T_\prelie)=&\twool {\prec'}{\prec'}312
-\twool {\prec'}{\prec}312
-\twoor {\prec'}1{\ast'}23
+\twool {\prec}{\prec'}213,
\nonumber\\
\su_2(T_\prelie)=&\twool {\prec'}{\succ'}312
-\twool {\prec'}{\succ}312
-\twoor {\succ'}1{\prec'}23
+\twool {\succ}{\ast'}213,
\nonumber\\
\su_3(T_\prelie)=&\twool {\succ'}{\ast'}312
-\twool {\succ'}{\ast}312
-\twoor {\succ'}1{\succ'}23
+\twool {\prec}{\succ'}213.
\nonumber
\end{align}}}
From the above nine equations, we can obtain the $\LDend$ operad relations.
Since the order from the biggest to the smallest is $\succ', \,\prec',\,\succ,\,\prec$, according to the leaf-lexicographic-path order,
the nine leading monomials of $\LDend$ operad are
 {\small{\begin{align}
\lbar{\su_1(\R_\prelie)}=&\twool {\prec}{\prec}312,\quad
\lbar{\su_2(\R_\prelie)}=\twool {\prec}{\succ}312,\quad
\lbar{\su_3(\R_\prelie)}=\twool {\succ}{\succ}312,\quad\label{eq:leadld1}\\
\lbar{\su_1(S_\prelie)}=&\twool {\prec}{\prec'}312,\quad
\lbar{\su_2(S_\prelie)}=\twool {\prec}{\succ'}312,\quad
\lbar{\su_3(S_\prelie)}=\twool {\succ}{\succ'}312,\label{eq:leadld2}\quad\\
\lbar{\su_1(T_\prelie)}=&\twool {\prec'}{\prec'}312,\quad
\lbar{\su_2(T_\prelie)}=\twool {\prec'}{\succ'}312,\quad
\lbar{\su_3(T_\prelie)}=\twool {\succ'}{\succ'}312.\label{eq:leadld3}
\end{align}}}
Hence, by a direct computation together with~(\mref{eq:ovepl}), we obtain~\eqref{eq:ldendo}.
\end{proof}

We arrive at our main result in this subsection.
\begin{theorem}\label{thm:gsbases}
Assume that $\cal{T}_{\shuffle}(\widetilde{\cal V_{\prelie}})$ is equipped with the leaf-lexicographic-path order in which the elements of $\widetilde{\cal V_{\prelie}}$
are ordered from largest to smallest as $\succ', \,\prec',\,\succ,\,\prec $.
Then the set of relations $\su(R_\prelie) = R_{\mathop{\LDend}}$ of $\mathop{\LDend}$ operad is a Gr\"obner-Shirshov basis in $\bfk\cal T_{\shuffle}(\widetilde{\cal V_{\prelie}})$.
\end{theorem}

\begin{proof}
According to Remark \ref{lem:rs}, we utilize the method of rewriting systems to establish that $\su(R_{\prelie})$ constitutes a Gr\"obner-Shirshov basis.
There are sixteen ambiguities (equivalently, local forks) listed in~\eqref{eq:ldendo}.
To illustrate the rewriting process, we focus on the monomial $\lbar{\su_1(T)}$ in~\eqref{eq:ldendo} and apply the defining relation of the operad $\mathop{\LDend}$. The treatment of other cases proceeds similarly and will be omitted for brevity.

In fact, the diamond (with respect to $\lbar{\su_1(T)}$) for the operad $\su(\prelie )=\LDend$ is the disuccessor of Figure~\ref{fig:diplie}. Indeed, on the one hand, for the divisor that shares its root with $\lbar{\su_1(T)}$, we obtain
\allowdisplaybreaks{\begin{align*}
&\threeoll {\succ'}{\succ'}4{\succ'}312\to \llbar{\threeoY {\succ'}{\succ'}{\ast'}1234}+\threeoll {\succ'}{\succ'}3{\succ'}412-\threeoY {\succ'}{\succ'}{\ast}1234\\
&\to \threeorr {\succ'}1{\ast'}2{\ast'}34+\threeolr {\succ'}{\succ'}21{\ast'}34-\threeorr {\succ'}1{\ast}2{\ast'}34+\threeoll {\succ'}{\succ'}3{\succ'}412-\llbar{\threeoY {\succ'}{\succ'}{\ast}1234}\\
&\to \threeorr {\succ'}1{\ast'}2{\ast'}34+\threeolr {\succ'}{\succ'}21{\ast'}34-\threeorr {\succ'}1{\ast}2{\ast'}34+\llbar{\threeoll {\succ'}{\succ'}3{\succ'}412}-\threeorr {\succ'}1{\ast'}2{\ast}34\\
&-\threeolr {\succ'}{\succ'}21{\ast}34+\threeorr {\succ'}1{\ast}2{\ast}34\\
 &\to \threeorr {\succ'}1{\ast'}2{\ast'}34+\threeolr {\succ'}{\succ'}21{\ast'}34-\threeorr {\succ'}1{\ast}2{\ast'}34+\threeolr {\succ'}{\prec}31{\ast'}24+\llbar{\threeoll {\succ'}{\succ'}3{\succ'}214}\\
&-\threeolr {\succ'}{\succ'}31{\ast}24-\threeorr {\succ'}1{\ast'}2{\ast}34-\threeolr {\succ'}{\succ'}21{\ast}34+\threeorr {\succ'}1{\ast}2{\ast}34\\
 &\to \threeorr {\succ'}1{\ast'}2{\ast'}34+\threeolr {\succ'}{\succ'}21{\ast'}34-\threeorr {\succ'}1{\ast}2{\ast'}34+\llbar{\threeolr {\succ'}{\succ'}31{\ast'}24}+\threeoY {\succ'}{\succ'}{\ast'}1423\\
 &+\threeoll {\succ'}{\succ'}2{\succ'}314-\threeoY {\succ'}{\succ'}{\ast}1423-{\threeolr {\succ'}{\succ'}31{\ast}24}-\threeorr {\succ'}1{\ast'}2{\ast}34-\threeolr {\succ'}{\succ'}21{\ast}34+\threeorr {\succ'}1{\ast}2{\ast}34\\
 &\to \threeorr {\succ'}1{\ast'}2{\ast'}34+\threeolr {\succ'}{\succ'}21{\ast'}34-\threeorr {\succ'}1{\ast}2{\ast'}34+\threeorl {\succ'}1{\ast'}{\ast'}324+\threeoY {\succ'}{\succ'}{\ast'}1324-\threeorl {\succ'}1{\ast}{\ast'}324\\
 &+\threeoY {\succ'}{\succ'}{\ast'}1423+\threeoll {\succ'}{\succ'}2{\succ'}314-\threeoY {\succ'}{\succ'}{\ast}1423-\llbar{\threeolr {\succ'}{\succ'}31{\ast}24}-\threeorr {\succ'}1{\ast'}2{\ast}34\\
 &-\threeolr {\succ'}{\succ'}21{\ast}34+\threeorr {\succ'}1{\ast}2{\ast}34\\
&\to \threeorr {\succ'}1{\ast'}2{\ast'}34+\threeolr {\succ'}{\succ'}21{\ast'}34-\threeorr {\succ'}1{\ast}2{\ast'}34+\threeorl {\succ'}1{\ast'}{\ast'}324+\threeoY {\succ'}{\succ'}{\ast'}1324-\threeorl {\succ'}1{\ast}{\ast'}324\\
 &+\threeoY {\succ'}{\succ'}{\ast'}1423+\threeoll {\succ'}{\succ'}2{\succ'}314-\threeoY {\succ'}{\succ'}{\ast}1423-\llbar{\threeorl {\succ'}1{\ast'}{\ast}324}-\threeoY {\succ'}{\succ'}{\ast}1324+\threeorl {\succ'}1{\ast}{\ast}324\\
 &-\threeorr {\succ'}1{\ast'}2{\ast}34-\threeolr {\succ'}{\succ'}21{\ast}34+\threeorr {\succ'}1{\ast}2{\ast}34\\
 &\to \threeorr {\succ'}1{\ast'}2{\ast'}34+\threeolr {\succ'}{\succ'}21{\ast'}34-\threeorr {\succ'}1{\ast}2{\ast'}34+\threeorl {\succ'}1{\ast'}{\ast'}324+\threeoY {\succ'}{\succ'}{\ast'}1324-\threeorl {\succ'}1{\ast}{\ast'}324\\
 &+\threeoY {\succ'}{\succ'}{\ast'}1423+\threeoll {\succ'}{\succ'}2{\succ'}314-\threeoY {\succ'}{\succ'}{\ast}1423-\threeorl {\succ'}1{\ast}{\ast'}423-\underline{\threeorr {\succ'}1{\ast}2{\ast}34}+\threeorl {\succ'}1{\ast}{\ast}423\\
 &-\threeoY {\succ'}{\succ'}{\ast}1324+\threeorl {\succ'}1{\ast}{\ast}324-\threeorr {\succ'}1{\ast'}2{\ast}34-\threeolr {\succ'}{\succ'}21{\ast}34+\underline{\threeorr {\succ'}1{\ast}2{\ast}34}\\
\end{align*}} 
Here, the two underlined terms cancel.
On the other hand, when the divisor does not share its root with $\lbar{\su_1(T)}$, we obtain
\allowdisplaybreaks{\begin{align*}
  &\threeoll {\succ'}{\succ'}4{\succ'}312\to \threeolr {\succ'}{\succ'}41{\ast'}23+\llbar{\threeoll {\succ'}{\succ'}4{\succ'}213}-\threeolr {\succ'}{\succ'}41{\ast}23\\
  &\to \llbar{\threeolr {\succ'}{\succ'}41{\ast'}23}+\threeoY {\succ'}{\succ'}{\ast'}1324+\threeoll {\succ'}{\succ'}2{\succ'}413-\threeoY {\succ'}{\succ'}{\ast}1324 -\threeolr {\succ'}{\succ'}41{\ast}23\\
  &\to \threeorl {\succ'}1{\ast'}{\ast'}423+\threeoY {\succ'}{\succ'}{\ast'}1423-\threeorl {\succ'}1{\ast}{\ast'}423
  +\threeoY {\succ'}{\succ'}{\ast'}1324+\threeoll {\succ'}{\succ'}2{\succ'}413\\
  &-\threeoY {\succ'}{\succ'}{\ast}1324 -\llbar{\threeolr {\succ'}{\succ'}41{\ast}23}\\
  &\to \llbar{\threeorl {\succ'}1{\ast'}{\ast'}423}+\threeoY {\succ'}{\succ'}{\ast'}1423-\threeorl {\succ'}1{\ast}{\ast'}423
  +\threeoY {\succ'}{\succ'}{\ast'}1324+\threeoll {\succ'}{\succ'}2{\succ'}413\\
  &-\threeoY {\succ'}{\succ'}{\ast}1324 -\threeorl {\succ'}1{\ast'}{\ast}423-\threeoY {\succ'}{\succ'}{\ast}1423+\threeorl {\succ'}1{\ast}{\ast}423\\
  &\to \threeorr {\succ'}1{\ast'}2{\ast'}34+\threeorl {\succ'}1{\ast'}{\ast'}324-\threeorr {\succ'}1{\ast'}2{\ast}34+\threeoY {\succ'}{\succ'}{\ast'}1423-\threeorl {\succ'}1{\ast}{\ast}423\\
  & +\threeoY {\succ'}{\succ'}{\ast'}1324+\threeoll {\succ'}{\succ'}2{\succ'}413-\threeoY {\succ'}{\succ'}{\ast}1324 -\llbar{\threeorl {\succ'}1{\ast'}{\ast}423}-\threeoY {\succ'}{\succ'}{\ast}1423+\threeorl {\succ'}1{\ast}{\ast}423\\
  &\to \threeorr {\succ'}1{\ast'}2{\ast'}34+\threeorl {\succ'}1{\ast'}{\ast'}324-\threeorr {\succ'}1{\ast'}2{\ast}34+\threeoY {\succ'}{\succ'}{\ast'}1423-\threeorl {\succ'}1{\ast}{\ast'}423+\threeoY {\succ'}{\succ'}{\ast'}1324\\
  &+\llbar{\threeoll {\succ'}{\succ'}2{\succ'}413}-\threeoY {\succ'}{\succ'}{\ast}1324
  -\threeorr {\succ'}1{\ast}2{\ast'}34-\threeorl {\succ'}1{\ast}{\ast'}324+\threeorl {\succ'}1{\ast}{\ast}324\\
  & -\threeoY {\succ'}{\succ'}{\ast}1423+\threeorl {\succ'}1{\ast}{\ast}423\\
  &\to \threeorr {\succ'}1{\ast'}2{\ast'}34+\threeorl {\succ'}1{\ast'}{\ast'}324-\threeorr {\succ'}1{\ast'}2{\ast}34+\threeoY {\succ'}{\succ'}{\ast'}1423-\threeorl {\succ'}1{\ast}{\ast'}423+\threeoY {\succ'}{\succ'}{\ast'}1324\\
  &+\threeolr {\succ'}{\succ'}21{\ast'}34+\threeoll {\succ'}{\succ'}2{\succ'}314-\threeolr {\succ'}{\succ'}21{\ast}34
  -\threeoY {\succ'}{\succ'}{\ast}1324-\threeorr {\succ'}1{\ast}2{\ast'}34\\
  & -\threeorl {\succ'}1{\ast}{\ast'}324+\threeorl {\succ'}1{\ast}{\ast}324-\threeoY {\succ'}{\succ'}{\ast}1423+\threeorl {\succ'}1{\ast}{\ast}423.
\end{align*}}
We observe that the $i$-th term in the expansion of the left hand side matches the $\sigma(i)$-th term in the expansion of the right hand side, where $\sigma$ is the following permutation of order $15$:
\begin{equation*}
\begin{pmatrix}
     i \\
     \sigma(i)
\end{pmatrix}
=
\begin{pmatrix}
1 & 2 & 3  & 4 & 5 & 6  & 7 & 8 &  9 & 10 & 11 & 12 & 13 & 14 & 15 &\\
1 & 7 & 11 & 2 & 6 & 12 & 4 & 8 & 14 & 5 &  15 & 10 & 13 &  3 & 9 &
\end{pmatrix}.
\end{equation*}
Thus  the local fork  $\lbar{\su_1(T)}$ is confluent. This completes the proof.
\end{proof}

\begin{proof}
By Theorem~\mref{thm:lipre}, the set $\R_\prelie$ is a Gr\"obner-Shirshov basis, which has diamond in Figure~\ref{fig:diplie}.
Similar to the one of Theorem~\mref{thm:lipre}, we claim that the diamond for the operad $\su(R_\prelie )=\LDend$ is the disuccessor of Figure~\ref{fig:diplie}.
\end{proof}

\begin{coro}
The operad $\LDend$ is Koszul.
\end{coro}
\begin{proof}
It follows from the result in ~\cite[Theorem 8.3.1]{LV12}.
\end{proof}

\section{Free L-dendriform algebras}
\mlabel{sect:free} In this section, we first present three linear bases for the free L-dendriform algebra: one obtained via the Gr\"obner-Shirshov method, and two others, $\lde$ and $\lw$, constructed through the free two-magma algebra and specialized typed decorated planar rooted trees. 
The following diagram illustrates the underlying idea.

$$\xymatrix{
\bfk\cal{T}_\shuffle(\{\succ', \,\prec',\,\succ,\,\prec\})(X) \ar@{->>}[d]\ar@{->}[rr]^{~~~~~~~\quad\quad\simeq}_{\quad\quad\quad {\rm two-magma\ algebra}}&& \bfk {\rm Mag_2}(X)\ar@{->>}[d]\ar@{->}[rr]^{\phi~\simeq}_{ {\rm two-magma\ algebra}}&&(\bfk \PMTN, \gleft , \gright)\ar@{->>}[d]\\
\LDend(X)=\bfk\operatorname{Irr}({R_{\LDend}})(X)\ar@{->}[rr]^{\quad\quad\simeq}_{\quad\quad{\rm L-dendriform ~~algebras}} \ar `d/20pt[r] `[rrrr]^{\simeq}_{{\rm L-dendriform ~~algebras}} [rrrr]&& \bfk \lde= {\rm Mag_2}(X)/I \ar@{->}[rr]^{\simeq}_{ \quad\quad{\rm L-dendriform ~~algebras}} && (\bfk\lw, \lp,\rp)\\
}
$$

A two-magma is a set  equipped with two binary operations, denoted by $\lp$ and $\rp$, without any additional relations. Denote $\fma$ the free two-magma on a set $\Omegax$.
Similar to the case of magma~\cite[Section 6]{DL}, each element of $\fma$ can be uniquely represented as $x\in X$ or in the form
\begin{equation}\label{eq:magun}
a = \Big(\big(((x \diamond_1 a_1) \diamond_2 a_2) \cdots\big) \diamond_n a_n\Big), \quad \text{ where }\, n\geq 1,~~ x \in \Omegax, ~~ \diamond_i \in \{\lp, \rp\},~~ a_i\in \fma.
\end{equation}
The {\bf degree} $\deg(a)$ is defined as the total number of $X$ occurring in $a$ counting the repetitions.
This representation captures the structure of the free two-magma as a sequence of nested applications of the operations $\lp$ and $\rp$ to the generating elements in $\Omegax$ and previously constructed elements of $\fma$.
The \textbf{free L-dendriform algebra} $\fna$ on a set $\Omegax$ is obtained as the quotient of the free two-magma algebra $\bfk\fma$ modulo the ideal $\ildend$  generated by all elements
\begin{equation}\label{eq:defld}
\begin{split}
(x \succ y) \prec z+x \succ(z \succ y)=&~x \succ(y \prec  z)+(x \ast z) \succ y,\\
(x \prec y) \prec z+ x\prec(z\ast  y)=&~x \prec(y \ast z)+(x \prec z) \prec y, \quad \text{ for all} ~~x, y,z\in \fma
\end{split}
\end{equation}
where $x \ast y=x \prec y+x \succ y$.

We now translate elements of $\fma$ into typed decorated planar rooted trees---namely, the planar analogue of the typed decorated rooted trees introduced in~\cite{Lo}.

\begin{defn}
Let $X$ and $\Omega$ be two nonempty sets.
\begin{enumerate}
\item An {\bf $X$-decorated $\Omega$-typed planar rooted tree} $t$ is a planar rooted tree,  whose vertices (internal vertexes and leaves) are decorated by elements in $X$ and whose edges are decorated by elements in $\Omega$. For simplicity, we refer to
$t$ as a {\bf typed decorated planar rooted tree} when the symbols $X$ and $\Omega$ do not need to be emphasized.

\item The degree $\deg(t)$ is defined as the total number of its vertices.

\item Denote the set of $X$-decorated $\{\lp, \rp\}$-typed planar rooted trees by \PMTN.
\end{enumerate}
\end{defn}

Let
\[
t_1, \dots, t_n\in \PMTN, \quad \diamond_1, \dots, \diamond_n\in \{\lp, \rp\}, \quad x\in X, \quad n\geq 0.
\]
Define
$B_x^+(\diamond_1,t_1, \dots, \diamond_n,t_n)$ to be
the typed decorated planar rooted tree in \PMTN\ obtained by grafting the $t_1, \dots, t_n$ (from left to right) onto a common new root decorated by $x$, where the edge connecting this new root to the root of each $t_i$ is
decorated by $ \diamond_i$.
Intuitively,
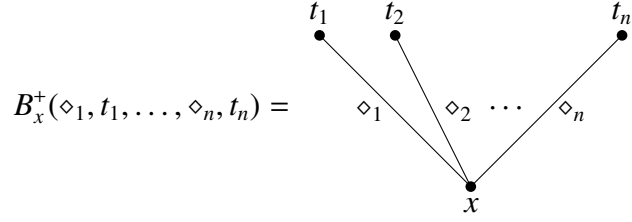
\[
B_x^+(\diamond_1,t_1, \dots, \diamond_n,t_n)=\begin{tikzpicture}
		\coordinate (z) at (0,0);
		\coordinate (y1) at (-2,2);
		\coordinate (y2) at (-1,2);
		\coordinate (yn) at (2,2);
		\draw (z) -- (y1) node[midway, left] {$\diamond_1$};
		\draw (z) -- (y2) node[midway,right] {$\diamond_2$};
		\draw (z) -- (yn) node[midway,right] {$\diamond_n$};
		\fill (y1) circle (2pt) node[above] {$t_1$};
		\fill (y2) circle (2pt) node[above] {$t_2$};
		\fill (z) circle (2pt) node[below] {$x$};
		\fill (yn) circle (2pt) node[above] {$t_n$};
		\node at (0.5,1) {$\cdots$};.
\end{tikzpicture}
\]
Here we use the convention that $\bullet_x = B_x^+(1)$, where $1$ is the empty forest.
Conversely, each typed decorated planar rooted tree in \PMTN\ can be uniquely represented as
\begin{equation}
B_x^+(\diamond_1,t_1, \dots, \diamond_n,t_n) \quad \text{ for some } x \in \Omegax, \quad \diamond_1, \dots, \diamond_n\in \{\lp, \rp\}, \quad t_1, \dots, t_n \in \PMTN, \quad n\geq 0.
\mlabel{eq:tdtree}
\end{equation}
A {\bf subtree} of a typed decorated planar rooted tree $t=B_x^+(\diamond_1,t_1, \dots, \diamond_n,t_n)$ in \PMTN\ is either the tree $t$ itself, or a subtree of one of $t_1,\ldots,t_n$.

The space $\bfk\PMTN$ can be turned into a two-magma algebra, equipped with the following two multiplications
\begin{equation}
\begin{split}
s \gleft t:=&~ B_x^+(\diamond_1',s_1, \dots, \diamond_m',s_m)\gleft B_y^+(\diamond_1,t_1, \dots, \diamond_n,t_n)\\
:=&~ B_x^+\Big(\diamond_1',s_1, \dots, \diamond_m',s_m, \lp, B_y^+(\diamond_1,t_1, \dots, \diamond_n,t_n)\Big),
\end{split}
\mlabel{eq:clp}
\end{equation}
and
\begin{equation}
\begin{split}
s \gright t :=&~ B_x^+(\diamond_1',s_1, \dots, \diamond_m',s_m)\gright B_y^+(\diamond_1,t_1, \dots, \diamond_n,t_n)\\
:=&~ B_x^+\Big(\diamond_1',s_1, \dots, \diamond_m',s_m, \rp, B_y^+(\diamond_1,t_1, \dots, \diamond_n,t_n)\Big),
\end{split}
\mlabel{eq:crp}
\end{equation}
where $$s=B_x^+(\diamond_1',s_1, \dots, \diamond_m',s_m), \quad t=B_y^+(\diamond_1,t_1, \dots, \diamond_n,t_n)\in \PMTN.$$
Intuitively, the operation $s\gleft t$ (resp. $s\gright t$) is obtained by grafting $t$ to the root of $s$ on the right, and decorating the new edge by $\lp$ (resp. $\rp$).
These operations can be viewed as decorated variants of the right grafting operation $\swarrow$; the left counterpart $\searrow$ is commonly used in the context of infinitesimal Hopf algebras~\cite{Fo}.

The triple $(\bfk \PMTN, \gleft, \gright)$ is indeed a free two-magma algebra. To establish this, we define the following linear map inductively on the degree:
\begin{equation}
\begin{aligned}
\phi: \bfk\fma \ra& ~~\bfk\PMTN\\
\Big(\big(((x\diamond_1 a_1)\diamond_2 a_2 \cdots)\diamond_n a_n\big)\Big)\mapsto&~~B_x^+\Big(\diamond_1,\phi(a_1), \diamond_2,\phi(a_2),\ldots, \diamond_n,\phi(a_n) \Big).
\end{aligned}
\mlabel{eq:phi}
\end{equation}
In particular, taking $n=0$ yields $\phi(x) = B_x^+(1) = \bullet_x$ for $x\in X$. Let us expose an example for better understanding.

\begin{exam}
Let $X=\{x,y,z,w\}$ be a set. Then
\[
\phi\Big(\big(x\lp w\big)\rp \big((y\lp x)\rp z\big)\Big)=B_x^+\Big(\lp, w,\rp,B_y^+(\lp, x, \rp, z)\Big)
=\begin{tikzpicture}[scale=0.8,descr/.style={fill=white}]
		\tikzstyle{every node}=[thick,minimum size=3pt, inner sep=1pt]
		
		\node (y) at (3,1.5) {};
		\node (t) at (3,3) [fill=black, circle] {};
		\node (b0) at (1.5,1.5) [ fill=black,circle] {};
		\node (b00) at (1.5,1.2) {};
		\node (b2) at (0.2,3) [fill=black,circle]  {};
		\node (cd) at (-0.25,1) {};
        \node (a2) at (-1.5,1.5) [fill=black,circle] {};
		\node (a0) at (0,0) [fill=black,circle] {};
		
		\node[above ] at (t) {$z$};
	   \node[above right] at (b2) {$x$};
		\node[above right] at (b00) {$y$};
         \node[above right] at (a2) {$w$};
		\node[below] at (a0) {$x$};
\node[above ] at (-0.7,0.7) {$\lp$};
\node[above ] at (0.7,0.7) {$\rp$};
\node[above ] at (0.7,2) {$\lp$};
\node[above ] at (2,2) {$\rp$};

		\draw (t) -- (b0);
	     \draw (b0) -- (b2);
        \draw (a0) -- (a2);
		\draw (a0) -- (b0);
\end{tikzpicture}.\]
\end{exam}

\begin{prop}\mlabel{prop:tmi}
The linear map $\phi$ in~\eqref{eq:phi} is an isomorphism of two-magma algebras.
\end{prop}

\begin{proof}
The bijectivity of $\phi$ follows from its explicit inverse:
\begin{equation*}
\begin{aligned}
\phi^{-1}: \bfk\PMTN \ra& \bfk\fma~~\\
B_x^+\Big(\diamond_1,t_1, \diamond_2,t_2,\ldots, \diamond_n,t_n \Big) \mapsto&~~\Big(\big(((x\diamond_1 \phi^{-1}(t_1))\diamond_2 \phi^{-1}(t_2) \cdots)\diamond_n \phi^{-1}(t_n)\big)\Big),
\end{aligned}
\mlabel{eq:phi0}
\end{equation*}
where the map is defined inductively on the degree of typed decorated planar rooted trees.
We are left to show that
$$\phi(a\diamond b) = \phi(a)\gd_\diamond \phi(b),\quad \text{ for } a, b\in \bfk\fma, \quad \diamond\in \{\lp,\rp\}.$$
Express the element $a$ in the form given by~\eqref{eq:magun}.
Then, by the definition of $\phi$, we obtain:
\begin{align*}
\phi(a\diamond b)=&~\phi\Bigg(\Big(\Big(\big(((x\diamond_1 a_1)\diamond_2 a_2 \cdots)\diamond_n a_n\big)\Big)\diamond b\Big)\Bigg)\\
=&~B_x^+\Big(\diamond_1,\phi(a_1), \diamond_2,\phi(a_2),\ldots, \diamond_n,\phi(a_n),\diamond, \phi(b) \Big)\\
=&~B_x^+\Big(\diamond_1,\phi(a_1), \diamond_2,\phi(a_2),\ldots, \diamond_n,\phi(a_n)\Big) \gd_\diamond \phi(b) \\
=&~\phi(a) \gd_\diamond \phi(b).
\end{align*}
This completes the proof.
\end{proof}

Note that the operad of two-magma algebras is precisely the free shuffle operad $\bfk \cal{T}_\shuffle(\{\succ', \,\prec',\,\succ,\,\prec\})$.
So by \cite[Proposition 5.2.6]{LV12}, we have the following isomorphism of two-magma algebras
\begin{equation}
\bfk\cal{T}_\shuffle(\{\succ', \,\prec',\,\succ,\,\prec\})(X)\cong \bfk \fma.
\mlabel{eq:tfma}
\end{equation}
Under this isomorphism, the 9 relations in \eqref{eq:oldend} from the L-dendriform operad, when evaluated on
$x\leq y\leq z\in X$ (by the action of the group $S_3$,  such order is assumed) are as follows:
\begin{align} 
(x \prec y) \prec z = &~x \prec(y \ast z)+(x \prec z) \prec y- x\prec(z\ast  y),~~y\neq z ~(\text{if $y=z$, identity is trivial})\mlabel{eq:9rel1},\\
(x \succ y) \prec z=&~x \succ(y \prec  z)+(x \ast z) \succ y-x \succ(z \succ y),\mlabel{eq:9rel2}\\
(x\lp y) \lp z=&~x\lp(y \lp z)+(x \lp z) \rp y- (x\rp y) \lp z-x\lp(z \rp y),\mlabel{eq:9rel3}\\
(y \succ x) \prec z=&~y \succ(x \prec  z)+(y \ast z) \succ x-y \succ(z \succ x),\mlabel{eq:9rel4}\\
(y \prec x) \prec z = &~y \prec(x \ast z)+(y \prec z) \prec x- y\prec(z\ast  x),~~~~x\neq z ~(\text{if $x=z$, identity is trivial})\mlabel{eq:9rel5}\\
(y \rp x) \succ z=&~y \succ(x \succ z)+(y \succ z) \prec x-(y \lp x) \succ z-y \succ(z \prec  x),\mlabel{eq:9rel6}\\
z \succ(y \succ x) =&~ z \succ(x \prec  y)+(z \ast y) \succ x-(z \succ x) \prec y,\mlabel{eq:9rel7}\\
z \succ(y \prec  x) =&~(z \succ y) \prec x+z \succ(x \succ y)-(z \ast x) \succ y,\mlabel{eq:9rel8}\\
z \prec(y \rp x)=&~(z \prec y) \prec x+z\rp(x\ast y)-(z\rp x)\rp y-z \prec(y \lp x),\mlabel{eq:9rel9}\\
&~~~~\quad\quad\quad\quad x\neq y ~(\text{if $x=y$, identity is trivial})\nonumber
\end{align}
where the nine leading monomials are
\begin{align*}
(x\rp y)\rp z, \quad(x\lp y)\rp z,\quad  (x\lp y)\lp z,\\
(y\lp x)\rp z, \quad(y\rp x)\rp z,\quad (y\rp x)\lp z, \\
z\lp (y\lp x), \quad z\lp (y\rp x), \quad z\rp (y\rp x).
\end{align*}
Among these nine leading monomials, some monomials share the same operadic structure:
\begin{itemize}
\item The first and fifth are both of the form $\rp\circ_1 \rp$.

\item The second and fourth share the form $\rp\circ_1 \lp$.
\end{itemize}
As a result, these $9$ monomials collapse into $7$ distinct types up to operadic composition, corresponding precisely to the 7 cases listed in Definition~\ref{defn:leadm} and Remark~\mref{rem:leadm}.
To illustrate the normal forms resulting from the rewriting system defined by~(\mref{eq:9rel1})-(\mref{eq:9rel9}), we present an example.

\begin{exam}
Consider the first type $(x\rp y)\rp z$, where the parentheses are grouped to the left.
Replacing the variables $x,y,z$ by
\[
(x\rp x_1)\lp x_2, \quad (y\rp y_1)\lp y_2, \quad (z\rp z_1)\lp z_2,
\]
where $x,y,z, x_i,y_i,z_i\in X$ with $x<y <z<x_1<y_1<z_1$,
we obtain
\begin{align*}
\ \Big(\big((x\rp x_1)\lp x_2\big)\rp \big((y\rp y_1)\lp y_2\big)\Big)\rp\big((z\rp z_1)\lp z_2\big)
\in\ \Big\{(a\rp b)\rp c~\Big|~a,b,c\in {\rm Mag_2}(X)\Big\}.
\end{align*}
Via the isomorphism $\phi$ in~\eqref{eq:phi}, this word corresponds to the typed decorated planar rooted tree
$$B^+_{x}\Big(\rp,x_1,\lp, x_2, \rp, B^+_{y}\Big(\rp,y_1,\lp, y_2\Big), \rp, B^+_{z}\Big(\rp,z_1,\lp, z_2\Big)\Big)\in \PMTN,$$
which matches the first case in~\eqref{eq:ldw1} and also corresponds to~\eqref{eq:ldwa}.

In contrast, let us consider another type $z\lp (y\lp x)$, where the parentheses are grouped to the right.
Replacing the variables $x,y,z$ by
\[
(z\rp z_1)\lp z_2, \quad (y\rp y_1)\lp y_2, \quad (x\rp x_1)\lp x_2,
\]
where $x,y,z, x_i,y_i,z_i\in X$ with $z<y <x<z_1<y_1<x_1<z_2<y_2<x_2$,
we get
\begin{align*}
\big((x\rp x_1)\lp x_2\big) \lp \Big( \big((y\rp y_1)\lp y_2\big)\lp \big((z\rp z_1)\lp z_2\big)\Big)
\in\ \Big\{a\lp (b\lp c)~\Big|~a,b,c\in {\rm Mag_2}(X)\Big\}.
\end{align*}
Under the isomorphism $\phi$ in~\eqref{eq:phi}, this word corresponds to the typed decorated planar rooted tree
$$B^+_{x}\Big(\rp,x_1,\lp, x_2, \lp, B^+_{y}\Big(\rp,y_1,\lp, y_2, \lp, B^+_{z}\Big(\rp,z_1,\lp, z_2\Big)\Big)\Big)\in \PMTN,$$
which corresponds to the third case in~\eqref{eq:ldw2}, as well as to~\eqref{eq:ldwb}.
\end{exam}

By systematizing the ideas illustrated in the above example, we introduce the notion of an L-dendriform word, which will serve as the basis elements of the free L-dendriform algebra. The point is that directly defining the L-dendriform word is challenging. Instead, we exploit the additional structural dimension present in the tree, using it as a tool to define the L-dendriform word.

\begin{defn}\label{defn:leadm}
Let $X$ be a well-ordered set. An element $a\in \fma$  is called an {\bf L-dendriform word}
if the corresponding typed decorated planar rooted tree  $\phi(a)$ in \PMTN\ does not contain any subtree of the form $B_x^+(\diamond_1,t_1, \dots, \diamond_n,t_n)$ satisfying either of the following conditions:
\begin{enumerate}
\item $t_{n-1}=B^+_y(\ast)$,\, $ t_n=B^+_z(\ast)$ and
\begin{equation}\label{eq:ldw1}
(\diamond_{n-1}, \diamond_{n})=
\begin{cases}
(\rp,\rp) , & \mbox{if }  x, y\leq z\in X \text{ and } y\neq z\\
(\lp,\rp), & \mbox{if }  x, y\leq z\in X \\
(\lp,\lp), & \mbox{if  } x\leq y\leq z\in X \\
(\rp,\lp), & \mbox{if  }  y\leq x\leq z\in X,
\end{cases}
\end{equation}

\item $t_{n}=B^+_y(\ast, \diamond, B^+_z(\ast))$ with $\diamond\in \{\lp,\rp\}$ and
\begin{equation}\label{eq:ldw2}
(\diamond_n, \diamond)=  \begin{cases} (\lp,\lp) , & \mbox{if } z\leq y\leq x\in X\\ (\lp,\rp), & \mbox{if } z\leq y\leq x\in X\\  (\rp,\rp), & \mbox{if } z\leq y\leq x\in X \text{ and } x\neq y.\end{cases}
\end{equation}
\end{enumerate}
Here, the symbol $\ast$ denotes an arbitrary subexpression chosen so that the entire expression is well-defined in the sense of~(\mref{eq:tdtree}).
\end{defn}

Let $\lde$ denote the set of L-dendriform words, and define
\begin{equation}
\lw :=\{\phi(a)~|~a \in\lde\}.
\mlabel{eq:lwt}
\end{equation}

\begin{remark}\label{rem:leadm}
An L-dendriform word $a\in\lde$ can be viewed as its typed decorated planar rooted tree  $\phi(a) $ in \PMTN\  does not contain any subtree of the form
\begin{align}
&B_x^+\Big(\ast,\rp, B_y^+\big(\ast \big),\rp,B_z^+\big(\ast\big)\Big),\quad\text{ with some }\, x,y\leq z\in X \text{ and } y\neq z \label{eq:ldwa}\\
&B_x^+\Big(\ast ,\lp, B_y^+\big(\ast \big),\rp,B_z^+\big(\ast \big) \Big),\quad\text{ with some }\, x,y\leq z\in X\nonumber\\
& B_x^+\Big(\ast ,\lp, B_y^+\big(\ast  \big),\lp,B_z^+\big(\ast \big)\Big), \quad \text{ with some }\, x\leq y\leq z\in X\nonumber\\
& B_x^+\Big(\ast ,\rp, B_y^+\big(\ast  \big),\lp,B_z^+\big(\ast \big) \Big), \quad \text{ with some }\, y\leq x\leq z\in X\nonumber
\end{align}
and
\begin{align}
&B_x^+\Big(\ast ,\lp, B_y^+\big(\ast ,\lp,B_z^+\big(\ast \big)\big) \Big),\quad\text{ with some }\, z\leq y\leq x\in X\label{eq:ldwb}\\
&B_x^+\Big(\ast ,\lp, B_y^+\big(\ast ,\rp,B_z^+\big(\ast \big)\big) \Big),\quad\text{ with some }\, z\leq y\leq x\in X\nonumber\\
&B_x^+\Big(\ast ,\rp, B_y^+\big(\ast ,\rp,B_z^+\big(\ast \big)\big) \Big),\nonumber \quad \text{ with some }\, z\leq y\leq x\in X \text{ and } x\neq y.
\end{align}
\end{remark}

We now arrive at one of the main results of this section, which provides three
\bfk-bases for the free L-dendriform algebra.

\begin{theorem}\mlabel{thm:gsbld}
Let $X$ be a well-ordered set.
\begin{enumerate}
\item The set
\begin{equation}
\begin{split}
\operatorname{Irr}({R_{\LDend}}) :=&~ \cal{T}_\shuffle\Big(\{\succ', \,\prec',\,\succ,\,\prec\}\Big) \\
& \Big\backslash\Big\{\left.q\right|_{\bar{s}} \mid q \in \cal{T}(\{\succ', \,\prec',\,\succ,\,\prec\})^{\star_3},  s \in \su(\{\R_\prelie,S_\prelie,T_\prelie\})\Big\}
\end{split}
\label{eq:irr}
\end{equation}
is a $\bfk$-basis of the operad $\LDend$.  \mlabel{it:gsbld1}

\item The $\lde$ and $\lw$ are two \bfk-bases of the free L-dendriform algebra $\LDend(X)$:
$$\LDend(X)\cong\bfk\lde\cong \bfk\lw. $$
\mlabel{it:gsbld2}
\end{enumerate}
\end{theorem}

\begin{proof}
\ref{it:gsbld1} It follows directly from Lemma~\mref{thm:cdl} and Theorem~\mref{thm:gsbases}.

\ref{it:gsbld2} By~(\mref{eq:tfma}), the set $\cal{T}_\shuffle\Big(\{\succ', \,\prec',\,\succ,\,\prec\}\Big)(X)$ is bijective to the set $\fma$. Further, the leading monomials $\bar{s}$ in~(\mref{eq:irr}) are bijective to the subtrees listed in Definition~\mref{defn:leadm}.
Thus it follows from~(\mref{eq:irr}) and Definition~\mref{defn:leadm} that
the $\operatorname{Irr}({R_{\LDend}})(X)$ is bijective to the set $\lde$, and so
$$\LDend(X) \overset{\ref{it:gsbld1}}{\cong} \bfk \operatorname{Irr}({R_{\LDend}})(X) \cong \bfk \lde.$$
The second required isomorphism is from~(\ref{eq:lwt}) and the fact that $\phi$ is an isomorphism.
\end{proof}

\smallskip

\noindent
{{\bf Acknowledgments.}
H. H. Zhang is supported by the Natural Science Basic Research Program of Shaanxi (2026JC-YBQN-0018),
the Young Talent Fund of Association for Science and Technology in Shaanxi, China (20250530), 
the Young Talent Fund of Association for Science and Technology in Yulin (20250711) and the Scientific Research Foundation of High-Level Talents of Yulin University (2025GK12).
X. Gao is supported by the National Natural Science Foundation of China (12571019), the Natural Science Foundation of Gansu Province (25JRRA644) and the Innovative Fundamental Research Group Project of Gansu Province (23JRRA684). 
Y. Y. Zhang is supported by  the Natural Science Foundation of China (12101183), the Natural Science Foundation of Henan (262300421229), the Postdoctoral Fellowship Program of CPSF under Grant Number (GZC20240406) and the Henan Provincial Selective Research Funding Program for Returned Scholars Studying Abroad (HNLX202613).

\medskip

\noindent
{\bf Declaration of interests. } The authors have no conflicts of interest to disclose.

\noindent
{\bf Data availability. } Data sharing is not applicable as no new data were created or analyzed.

\end{document}